\documentclass[leqno]{amsart}
\usepackage{amsmath, amsthm, amssymb, amstext}
\usepackage{xcolor}
\usepackage[left=2.5cm,right=2.5cm,top=1.5cm,bottom=1.5cm,includeheadfoot]{geometry}
\usepackage{hyperref,xcolor}
\hypersetup{
 pdfborder={0 0 0},
 colorlinks,
}
\usepackage{dsfont}
\usepackage{mathrsfs}  

\usepackage{mathtools}
\usepackage{color}

\usepackage{todonotes}
\usepackage{enumitem}
\allowdisplaybreaks

\newtheorem{theorem}{Theorem}

\newtheorem{lemma}[theorem]{Lemma}
\newtheorem{proposition}[theorem]{Proposition}

\newcommand{\N}{\mathbb{N}}

\newcommand{\derivada}{\,^{\bf H}\mathbf{D}_{0+}^{\alpha,\beta;\psi}}
\newcommand{\derivadaCT}{\,^{\bf H}\mathbf{D}_{T}^{\alpha,\beta;\psi}}
\newcommand{\hzero}{\mathbb{H}^{\alpha,\beta;\psi}_{\mathcal{H},0}}
\newcommand{\norma}[3][2]{\lVert #1\rVert^{#2}_{#3}}

\numberwithin{theorem}{section} \numberwithin{equation}{section}

\title[Existence of solutions for  a singular...]{Existence of solutions for  a singular double phase problem involving a $\psi$-Hilfer Fractional operator via Nehari Manifold}

\author[J. Vanterler da C. Sousa $^{*}$, Karla B. Lima, Leandro S. Tavares]{J. Vanterler da C. Sousa $^{*}$, Karla B. Lima,  Leandro S. Tavares}

\address[J. Vanterler da C. Sousa]{Center for Mathematics, Computing and Cognition, Federal University of ABC, Avenida dos Estados, 5001, Bairro Bangu, 09.210-580, Santo André, SP - Brazil}
\email{\tt jose.vanterler@ufabc.edu.br}

\address[Karla B. Lima]{Faculty of Exact Sciences and Technology, FACET, UFGD, Dourados, MS 79804-970, Brazil}
\email{\tt karlalima@ufgd.edu.br}

\address[Leandro S. Tavares]{Center of  Sciences and  Technology, Federal University of Cariri, Juazeiro do Norte, Brazil;}
\email{\tt leandro.tavares@ufca.edu.br}


\begin{document}
 \begin{abstract} 
\textcolor{blue}{ In this present paper, we investigate a new class of singular double phase $p$-Laplacian equation problems with a $\psi$-Hilfer fractional operator combined from a parametric term. Motivated by the fibering method using the Nehari manifold, we discuss the existence of at least two weak solutions to such problems when the parameter is small enough. Before attacking the main contribution, we discuss some results involving the energy functional and the Nehari manifold.}

\end{abstract}
\maketitle

\section{Introduction and motivation}

\textcolor{blue}{This paper is concerned with the fractional singular double phase problem
\begin{equation}\label{eq1}
\left\{ 
\begin{array}{ccc}
{\bf L}^{\alpha,\beta}_{\psi} u & = & a(x)u^{-\gamma}+\lambda u^{r-1}, \, \mbox{ in }\Omega=[0,T]\times[0,T] \\ 
u & =&0, \mbox{ on }\partial \Omega    
\end{array}%
\right.
\end{equation}
with 
\begin{equation*}
    {\bf L}^{\alpha,\beta}_{\psi}u:=\derivadaCT\left(\arrowvert\derivada u\arrowvert^{p-2}\,\derivada  u+ \mu(x)\arrowvert\derivada u\arrowvert^{q-2}\,\derivada u\right),
\end{equation*}
where $\derivadaCT(\cdot)$ and $\derivada(\cdot)$ denotes the  $\psi$-Hilfer fractional operator of order $\dfrac{1}{p}<\alpha<1$ and type $0\leq \beta\leq 1$ respectively and}
\begin{itemize}
    \item[\textbf{(H)}]
    \begin{itemize}
    \item[(i)] \textcolor{blue}{ $1<p<2$, $p<q<p^{*}_{\alpha}=\dfrac{2p}{2-\alpha p}$ and $0\leq \mu(\cdot)\in L^{\infty}(\Omega)$;}
    
    \item[(ii)] $0<\gamma<1$ and $q<r<p_{\alpha}^*$;
    \item[(iii)]$a\in L^\infty(\Omega)$ and $a(x)> 0$ for a.a. $x\in\Omega$, with $a\neq 0$.
\end{itemize}
\end{itemize}

Partial differential equations involving double phase operators arise in the study of  the  behaviour of strongly anisotropic materials, see for instance Zhikov \cite{Zhikov}. Zhikov found that the hardening properties of such materials can change drastically, which is known in the literature as the Lavrentiev’s phenomenon, see \cite{1Zhikov,2Zhikov,3Zhikov} for more details. In order to describe such phenomenon, in \cite{Zhikov} it was introduced the energy functional defined in a suitable Sobolev space, which can change its ellipticity, given by  
\begin{equation}\label{funmin}
\int_{\Omega} |\nabla u |^{p}  + a(x) |\nabla u |^{q} dx,
\end{equation}
where $\Omega \subset \mathbb{R}^N$ is a domain and $p,q >1$. The function $a$  is related to  mixture between two different materials with power hardening of rates $p$ and $q$, respectively. In the recent papers \cite{Baroni,1Baroni,Colombo,Colombo1} it was obtained regularity results for minimizers of the functional \eqref{funmin} with $q >p$ and  $a \geq 0.$

In the  recent years, there has been an increasing interest  in the study of variational problems with double phase operators, see for instance  \cite{Liu,Ok,Papageorgiou} and the references therein. For example  Liu and Dai in \cite{Liu} discussed some existence and multiplicity results for the problem 
\begin{equation*}
	\left\{
	\begin{array}{rcl}
-div( |\nabla u |^{p-2} \nabla u+ a(x) |\nabla u |^{q-2} \nabla u)
		&=&f(x,u) \mbox{ in } \Omega\mbox{,}\\
		u & = & 0\;\;\mbox{on}\;\;\partial \Omega ,
	\end{array}%
	\right. 
\end{equation*}
where $\Omega\subset \mathbb{R}^{N}$ is a bounded domain with smooth boundary, $N\geq 2$, $1<p<q<N$, $\frac{q}{p}<1+\frac{1}{N},\,a:\overline{\Omega}\rightarrow [0,\infty)$ is Lipschitz continuous and $f$ satisfies some conditions. \textcolor{blue}{ See also the work discussed by Ragusa and Tachikawa \cite{Ragusa}, on regularity for minimizers for functionals of double phase with variable exponents. }

\textcolor{blue}{In Papageorgiou et al. \cite{Papageorgiou} it was investigated the existence of positive solutions for the singular  double phase problem
\begin{equation*}
	\left\{
	\begin{array}{rcl}
	-\Delta_{p} u -div( \xi(z) |\nabla u |^{p-2} \nabla u) 
		&=&a(z) u^{-\gamma}+\lambda u^{r-1} \mbox{ in } \Omega\mbox{,}\\
		u & = & 0\;\;\mbox{on}\;\;\partial \Omega ,
	\end{array}%
	\right. 
\end{equation*}
where $\Omega\subset \mathbb{R}^{N}$ is a bounded domain with smooth boundary, $N \geq 2, $ $\xi\in L^{\infty}(\Omega)$, $a\in L^{\infty}$, $a\geq 0$, $\xi(z)>0$, $u\neq 0$, $1<q<p<r<p^{*}$, $0<\gamma<1$, $u\geq 0$ and $\lambda>0$.}

Cui and Sun in \cite{Cui}  considered the existence and multiplicity of solutions for the  blue problem with nonlinear boundary conditions given by 
\begin{equation*}
	\left\{
	\begin{array}{rcl}
	-div(|\nabla u |^{p-2} \nabla u+ \mu(x) |\nabla u |^{q-2} |\nabla| u)
		&=&f(x,u)-|u|^{p-2}u-\mu(x) |u|^{q-2} u \mbox{ in } \Omega\mbox{,}\\
		\left(|\nabla u |^{p-2} \nabla u + \mu(x) |\nabla u |^{q-2} |\nabla|u\right).\nu & = & g(x,u)\;\;\mbox{on}\;\;\partial \Omega ,
	\end{array}%
	\right. 
\end{equation*}
where $\Omega\subset \mathbb{R}^{N}$ is bounded with smooth boundary, $N\geq 2$, $1<p<q<N$, $\mu \in L^{\infty}(\Omega)$ with $\mu(x)\geq 0$ a.e. $x\in\Omega$, $\nu$ denotes the outer unit normal of $\Omega$ at the point $x\in\partial\Omega$, $f:\Omega\times\mathbb{R}\rightarrow\mathbb{R}$ and $g:\partial\Omega\times\mathbb{R}\rightarrow\mathbb{R}$ are Caratheodory functions. \textcolor{blue}{ In addition, it is worth mentioning another important work on the eigenvalue problem in the double phase variational variation structure
integrals performed by Colasuonno and Squassina \cite{Colasuonno}. On the other hand, Colombo and Giuseppe \cite{Colombo} did an excellent work on bounded minimisers of double phase variational integrals. For other interesting papers which considers  double phase problems we mention \cite{Bahrouni,Gasinski,Lei,Wulong} and the references therein.}

On the other hand, the theory of fractional differential equations is well consolidated due  to  important results, which can be found  for example in  \cite{Kilbas,Sousa,Vangipuram}. \textcolor{red}{An important fact is that  Fractional Differential Equations can be applied to study several models  in Medicine, Physics, Engineering, Mechanics and are interesting from the mathematical viewpoint,  see for example  \cite{Machado,Nemati,Silva1,Sousa22,You}}. For example in the recent paper    \cite{Sousa}, the authors introduced the $\psi$-Hilfer fractional operator and exhibited an wide class of examples. We also quote the reference  \cite{Sousa8}, where  it was considered  the construction of the space $\mathbb{H}^{\alpha,\beta;\psi}_{p}([0,T])$  which allowed to consider the variational approach for problems involving  $\psi$-Hilfer fractional operators.  

\textcolor{blue}{ In 2021 Suwan et al. \cite{Suwan}, obtained results of great importance in the field of differential equations with respect to the $\psi$-Caputo fractional operator. Through Dhage's fixed-point theory, they obtained the existence of a solution for the proposed problem. In this sense, in order to elucidate the results investigated, examples were discussed.}

 The pioneering work that addresses variational problems with $\psi$-Hilfer fractional operators via the Nehari manifold method is \cite{Sousa7}, whose authors    considered the  problem
\begin{equation}\label{nvm}
	\left\{
	\begin{array}{rcl}
	^{\bf H}{\bf D}_{T}^{\nu ,\eta ;\psi }\left( \left\vert ^{\bf H}{\bf D}_{0+}^{\nu ,\eta ;\psi }\xi\right\vert ^{p-2}\text{ }^{\bf H}{\bf D}_{0+}^{\nu ,\eta ;\psi }\xi\right)
		&=&\lambda |\xi|^{p-2} \xi+b(x) |\xi|^{q-1} \xi \mbox{ in } [0,T]\mbox{,}\\
	{\bf I}_{0+}^{\eta (\eta -1);\psi }\xi(0)  & = & {\bf I}_{T}^{\eta (\eta -1);\psi }\xi(T)=0 ,
	\end{array}%
	\right. 
\end{equation}
where $^{\bf H}{\bf D}_{T}^{\nu ,\eta ;\psi }(\cdot)$, $^{\bf H}{\bf D}_{0+}^{\nu ,\eta ;\psi } (\cdot)$ are the left-sided and right-sided $\psi$-Hilfer fractional operators  of order $\frac{1}{p}<\nu<1$ and type $\eta$ $(0\leq \eta\leq1)$ respectively, $1<q<p-1<\infty$, $b\in L^{\infty}(\Omega)$ and ${\bf I}_{0+}^{\eta (\eta -1);\psi }(\cdot)$, ${\bf I}_{T}^{\eta (\eta -1);\psi }(\cdot)$ are $\psi$-Riemann-Liouville fractional integrals of order $\eta (\eta -1)$ $(0\leq \eta\leq 1)$. In the reference \cite{Sousa6} it was discussed necessary and sufficient conditions for Eq.(\ref{nvm}) and investigated the bifurcation of solutions through the Nehari manifold technique. \textcolor{blue}{Several researchers have been working with variational problems involving fractional operators, such as: Sousa, Tavares, Ledesma, Pulido, Oliveira \cite{Sousa8,Sousa5,Sousa4}; Nyamoradi and Tayyebi \cite{Nyamoradi}, Ghanmi and Z. Zhang \cite{Ghanmi} and other researchers as per the works \cite{Jiao,Zhang,Zhao} and the references therein.}

\textcolor{blue}{Let $\theta=(\theta_{1},\theta_{2},\theta_{3})$, $T=(T_{1},T_{2},T_{3})$ and $\alpha=(\alpha_{1},\alpha_{2},\alpha_{3})$ where $0<\alpha_{1},\alpha_{2},\alpha_{3}<1$ with $\theta_{j}<T_{j}$, for all $j\in \left\{1,2,3 \right\}$. Also put $\Lambda=I_{1}\times I_{2}\times \times I_{3}=[\theta_{1},T_{1}]\times [\theta_{2},T_{2}]\times [\theta_{3},T_{3}]$, where $T_{1},T_{2},T_{3}$ and $\theta_{1},\theta_{2},\theta_{3}$ positive constants. Consider also $\psi(\cdot)$ be an increasing and positive monotone function on $(\theta_{1},T_{1}),(\theta_{2},T_{2}),(\theta_{3},T_{3})$, having a continuous derivative $\psi'(\cdot)$ on $(\theta_{1},T_{1}],(\theta_{2},T_{2}],(\theta_{3},T_{3}]$. The $\psi$-Riemann-Liouville fractional partial integrals of $u\in \mathscr{L}^{1}(\Lambda)$ of order $\alpha$ $(0<\alpha<1)$ are given by \cite{Srivastava}}
\begin{itemize}
    \item \textcolor{blue}{ 1-variable: right and left-sided
\begin{equation*}
    {\bf I}^{\alpha,\psi}_{\theta_{1}} u(x_{1})=\dfrac{1}{\Gamma(\alpha_{1})} \int_{\theta_{1}}^{x_{1}} \psi'(s_{1})(\psi(x_{1})- \psi(s_{1}))^{\alpha_{1}-1} u(s_{1}) ds_{1},\,\,to\,\,\theta_{1}<s_{1}<x_{1}
\end{equation*}
and
\begin{equation*}
    {\bf I}^{\alpha,\psi}_{T_{1}} u(x_{1})=\dfrac{1}{\Gamma(\alpha_{1})} \int_{x_{1}}^{T_{1}} \psi'(s_{1})(\psi(s_{1})- \psi(x_{1}))^{\alpha_{1}-1} u(s_{1}) ds_{1},\,\,to\,\,x_{1}<s_{1}<T_{1},
\end{equation*}
with $x_{1}\in[\theta_{1},T_{1}]$, respectively.}
    
\item \textcolor{blue}{ 3-variables: right and left-sided
\begin{eqnarray*}
    {\bf I}^{\alpha,\psi}_{\theta} u(x_{1},x_{2},x_{3})=\dfrac{1}{\Gamma(\alpha_{1})\Gamma(\alpha_{2})\Gamma(\alpha_{3})} \int_{\theta_{1}}^{x_{1}}
    \int_{\theta_{2}}^{x_{2}}
    \int_{\theta_{3}}^{x_{3}}
    \psi'(s_{1})\psi'(s_{2})\psi'(s_{3})
    (\psi(x_{1})- \psi(s_{1}))^{\alpha_{1}-1}\notag\\
    \times
    (\psi(x_{2})- \psi(s_{2}))^{\alpha_{2}-1}
    (\psi(x_{3})- \psi(s_{3}))^{\alpha_{3}-1}
    u(s_{1},s_{2},s_{3}) ds_{3}ds_{2}ds_{1},
\end{eqnarray*}
to $\theta_{1}<s_{1}<x_{1}, \theta_{2}<s_{2}<x_{2}, \theta_{3}<s_{3}<x_{3}$ and
\begin{eqnarray*}
    {\bf I}^{\alpha,\psi}_{T} u(x_{1},x_{2},x_{3})=\dfrac{1}{\Gamma(\alpha_{1})\Gamma(\alpha_{2})\Gamma(\alpha_{3})} \int_{x_{1}}^{T_{1}}
    \int_{x_{2}}^{T_{2}}
    \int_{x_{3}}^{T_{3}}
    \psi'(s_{1})\psi'(s_{2})\psi'(s_{3})
    (\psi(s_{1})-\psi(x_{1}))^{\alpha_{1}-1}\notag\\
    \times
    (\psi(s_{2})-\psi(x_{2}))^{\alpha_{2}-1}
    (\psi(s_{3})-\psi(x_{3}))^{\alpha_{3}-1}
    u(s_{1},s_{2},s_{3}) ds_{3}ds_{2}ds_{1},
\end{eqnarray*}
with $x_{1}<s_{1}<T_{1}, x_{2}<s_{2}<T_{2}, x_{3}<s_{3}<T_{3}$, $x_{1}\in[\theta_{1},T_{1}]$, $x_{2}\in[\theta_{2},T_{2}]$ and $x_{3}\in[\theta_{3},T_{3}]$, respectively.}
\end{itemize}

\textcolor{blue}{ On the other hand, let $u,\psi \in C^{n}(\Lambda)$ two functions such that $\psi$ is increasing and $\psi'(x_{j})\neq 0$ with $x_{j}\in[\theta_{j},T_{j}]$, $j\in \left\{1,2,3 \right\}$. The left and
right-sided $\psi$-Hilfer fractional partial derivative of $3$-variables of $u\in AC^{n}(\Lambda)$ of order $\alpha=(\alpha_{1},\alpha_{2},\alpha_{3})$ $(0<\alpha_{1},\alpha_{2},\alpha_{3}\leq 1)$ and type $\beta=(\beta_{1},\beta_{2},\beta_{3})$ where $0\leq\beta_{1},\beta_{2},\beta_{3}\leq 1$, are defined by \cite{Srivastava}
\begin{equation*}
{^{\mathbf H}{\bf D}}^{\alpha,\beta;\psi}_{\theta}u(x_{1},x_{2},x_{3})= {\bf I}^{\beta(1-\alpha),\psi}_{\theta} \Bigg(\frac{1}{\psi'(x_{1})\psi'(x_{2})\psi'(x_{3})} \Bigg(\frac{\partial^{3}} {\partial x_{1}\partial x_{2}\partial x_{3}}\Bigg) \Bigg) {\bf I}^{(1-\beta)(1-\alpha),\psi}_{\theta} u(x_{1},x_{2},x_{3})
\end{equation*}
and
\begin{equation*}
{^{\mathbf H}{\bf D}}^{\alpha,\beta;\psi}_{T}u(x_{1},x_{2},x_{3})= {\bf I}^{\beta(1-\alpha),\psi}_{T} \Bigg(-\frac{1}{\psi'(x_{1})\psi'(x_{2})\psi'(x_{3})} \Bigg(\frac{\partial^{3}} {\partial x_{1}\partial x_{2}\partial x_{3}}\Bigg) \Bigg) {\bf I}^{(1-\beta)(1-\alpha),\psi}_{T} u(x_{1},x_{2},x_{3}),
\end{equation*}}
where $\theta$ and $T$ are the \textcolor{red}{ same parameters presented in the definition of $\psi$-Riemann-Liouville fractional integrals ${\bf I}_{T}^{\alpha;\psi}(\cdot)$ and ${\bf I}_ {\theta}^{\alpha;\psi}(\cdot)$. For a study of $N$-variables, see \cite{Srivastava}.}

\textcolor{blue}{ Taking $\theta =0$ in the definition of ${^{\mathbf H}{\bf D}}^{\alpha,\beta;\psi}_{\theta}(\cdot)$, we have ${^{\mathbf H}{\bf D}}^{\alpha,\beta;\psi}_{0}(\cdot)$. During the paper we will use the following notation: ${^{\mathbf H}{\bf D}}^{\alpha,\beta;\psi}_{\theta} u(x_{1},x_{2},x_{3}):= {^{\mathbf H}{\bf D}}^{\alpha,\beta;\psi}_{\theta} u$, ${^{\mathbf H}{\bf D}}^{\alpha,\beta;\psi}_{T} u(x_{1},x_{2},x_{3}):= {^{\mathbf H}{\bf D}}^{\alpha,\beta;\psi}_{T} u$ and ${\bf I}_{\theta}^{\alpha;\psi}u(x_{1},x_{2},x_{3}):= {\bf I}_{\theta}^{\alpha;\psi}u$.}

 \textcolor{blue}{ A function $u\in \mathbb{H}^{\alpha,\beta;\psi}_{\mathcal{H},0}(\Omega)$ is said to be a weak solution if $u>0$ for a.a. $x\in\Omega$ and 
\begin{equation*}
\begin{split}
    \displaystyle\int_{\Omega}\left(\arrowvert\derivada u\arrowvert^{p-2}\,\derivada  u+ \mu(x)\arrowvert\derivada u\arrowvert^{q-2}\,\derivada u\right)\derivada h\mathrm{d}x\\
   =\displaystyle\int_{\Omega}a(x)u^{-\gamma}h\mathrm{d}x+\lambda\displaystyle\int_{\Omega} u^{r-1}hdx
\end{split}
\end{equation*}
is satisfied for all $h\in  \mathbb{H}^{\alpha, \beta,\psi}_{\mathcal{H},0}(\Omega)$.}

Let $A:\mathbb{H}^{\alpha,\beta;\psi}_{\mathcal{H},0}(\Omega)\rightarrow \mathbb{H}^{\alpha,\beta;\psi}_{\mathcal{H},0}(\Omega)^*$, where $\mathbb{H}^{\alpha,\beta;\psi}_{\mathcal{H},0}(\Omega)^*$ denotes the dual of $\mathbb{H}^{\alpha,\beta;\psi}_{\mathcal{H},0}(\Omega)$, be the nonlinear map defined by
\begin{equation*}
     \langle Au, \varphi\rangle \coloneqq \displaystyle\int_\Omega \left(\lvert\derivada u\rvert^{p-2}\derivada u+ \mu(x)\lvert\derivada u\rvert^{q-2}\derivada u\right)\derivada \varphi\mathrm{d}x,
\end{equation*}
for all $u,\varphi\in \mathbb{H}^{\alpha,\beta;\psi}_{\mathcal{H},0}(\Omega)$.

Consider the energy functional associated to problem Eq.(\ref{eq1}), say $\mathbf{E}_\lambda:\mathbb{H}^{\alpha,\beta;\psi}_{\mathcal{H},0}(\Omega)\rightarrow\mathbb{R}$ given by
\begin{equation*}
    \mathbf{E}_\lambda(u)=\frac{1}{p}\lVert \derivada u\rVert_{L^{p}(\Omega)}^{p} +\frac{1}{q}\lVert \derivada u\rVert_{L_{\mu}^{q}(\Omega)}^{q}-\frac{1}{1-\gamma}\displaystyle\int_\Omega a(x)|u|^{1-\gamma}\mathrm{d}x -\frac{\lambda}{r}\lVert u\rVert_{L^{r}(\Omega)}^{r}.
\end{equation*}

The main result in this paper is the following theorem:

\begin{theorem}\label{prin} Consider that the  hypotheses {\rm\textbf{(H)}} are satisfied. Then, there exists $\hat{\lambda}_{0}^*>0$ such that, for all $\lambda \in(0,\hat{\lambda}_{0}^*]$, problem {\rm(\ref{eq1})} has at least two weak solutions $u^*, v^*\in \mathbb{H}^{\alpha,\beta;\psi}_{\mathcal{H},0}(\Omega)$, with $  \mathbf{E}_\lambda(u^*)<0<  \mathbf{E}_\lambda(v^*)$.
\end{theorem}

\subsection{Remarks and Consequences.}
 
\begin{itemize}
       \item \textcolor{blue}{ A natural consequence, when investigating problems involving fractional operators, is to discuss the integer case, in other words, at the limit $\alpha\rightarrow 1$. However, in the investigated problem Eq.(\ref{eq1}), in addition to $\alpha$, it is also necessary to take $\psi(x)=x$ to obtain the classical problem in $\mathbb{R}^{2 }$, given by
    \begin{equation*}
\left\{ 
\begin{array}{ccc}
{\bf L} u & = & a(x)u^{-\gamma}+\lambda u^{r-1}, \, \mbox{ in }\Omega=[0,T]\times[0,T] \\ 
u & =&0, \mbox{ on }\partial \Omega    
\end{array}%
\right.
\end{equation*}
with
    \begin{equation*}
        {\bf L} u:= \sum_{i=1}^{2} \,\,
\dfrac{\partial}{\partial x_{i}}\left( \left\vert \dfrac{\partial u}{\partial x_{i}}\right\vert ^{p-2}\dfrac{\partial u}{\partial x_{i}}+ \mu(x) \left\vert \dfrac{\partial u}{\partial x_{i}}\right\vert ^{p-2}\dfrac{\partial u}{\partial x_{i}} \right)
\end{equation*}    
with the same conditions on $\phi$ as given in the problem (\ref{eq1}).}

\item \textcolor{blue}{ From the previous item, we have that all the results investigated here are valid for the entire case.}

\item \textcolor{blue}{ One of the advantages when working with fractional operators, in addition to the generalization provided by the parameters $\alpha$ and $\beta$, and the function $\psi(\cdot)$ in the mathematical sense, it provides a more detailed analysis when involves physical problems as highlighted at the beginning of the work, precisely because of the freedom to float $\frac{1}{p}<\alpha<1$. See, for example, the following papers where authors use fractional operators to discuss real problems and highlight the importance of operator order.}

\item \textcolor{blue}{ For example, taking $\psi(x)=x$ and the limit of $\beta\rightarrow 1$ we have that the problem is valid for the Riemann-Liouville fractional operator. On the other hand, taking $\psi(x)=x$ and the limit of $\beta \rightarrow 0$, we have the Caputo fractional operator context problem. However, the way the problem was proposed, as well as the space, here we cannot take the particular case of $\psi(x)=\ln x$, as it would not satisfy the conditions imposed on $\psi(x)$ as per in the definition of the $\psi$-Hilfer fractional operator with its respective space.}

\item \textcolor{blue}{ Based on the previous item, there is a way around this problem, which is to work in space $\mathbb{H}_{\mathcal{H},0}^{\alpha,\beta;\psi}(\Omega)$ with weight, where the weight is exactly $\psi'(\cdot)$.}

\item \textcolor{blue}{ One of the important consequences of this work is that it will provide results of singular double phase problems $p$-Laplacian equation problems with a $\psi$-Hilfer fractional operator for the literature, especially for the field of fractional differential equations, since problems of this type (see problem Eq.(\ref{eq1})), it is still very restricted and there are few tools and results. In this sense, there are numerous open issues that can be discussed.}

\end{itemize}

The rest of the paper is organized as follows: In section 2, we present some variational results which will be needed  to obtain the main results of this paper. In section 3, we investigate some important results that will be used in the  proof of  Theorem \ref{prin}.

\section{Preliminaries}

Let $\mathcal{H}: \Omega\times [0,\infty)\rightarrow[0,\infty)$ be the function defined by $\mathcal{H}(x,t)= t^p +\mu(x)t^q$.
Then, the weighted Musielak-Orlicz space $L^{\mathcal{H}}(\Omega)$ is defined by \cite{Colasuonno,Musi} as
\begin{align*}
L^{\mathcal{H}}(\Omega) = \{\, u:\Omega \rightarrow \mathbb{R} \mbox{ is measurable with } \rho^{\mathcal{H}}(u)<+\infty\}
\end{align*}
 with the Luxemburg norm
\begin{align*}
   \left\| u\right\|^{\mathcal{H}} = \inf\left\{\tau>0:\rho^{\mathcal{H}}\left(\frac{u}{\tau}\right)\leq 1\right\},
\end{align*}
where the modular function $\rho^{\mathcal{H}}:L^{\mathcal{H}}(\Omega)\rightarrow \mathbb{R}$ is given by
\begin{equation}\label{pl}
    \rho^{\mathcal{H}}\coloneqq \displaystyle\int_{\Omega}\mathcal{H}(x,|u|)\mathrm{d}x\equiv \displaystyle\int_{\Omega}\left(|u|^p+\mu(x)|u|^q\right)\mathrm{d}x.
\end{equation}

Moreover, we define the seminormed space
\begin{align*}
    L^{q}_{\mu}(\Omega) = \left\{u \,|\, u:\Omega \rightarrow \mathbb{R} \mbox{ is measurable and } \displaystyle\int_\Omega \mu(x)|u|^q\mathrm{d}x<+\infty\right\},
\end{align*}
which is endowed with the seminorm 
\begin{align*}
    \lVert u\rVert_{\mu} = \left(\displaystyle\int_\Omega \mu(x)|u|^q\right)^{\frac{1}{q}}.
\end{align*}

\textcolor{blue}{ The $(\psi,\mathcal{H})$-fractional space $\mathbb{H}_{\mathcal{H},0}^{\alpha,\beta;\psi}(\Omega)$ is defined by
\begin{align*}
 \mathbb{H}_{\mathcal{H},0}^{\alpha,\beta;\psi}(\Omega)   \equiv\left\{u\in  L^{\mathcal{H}}(\Omega) : |\derivada u|\in L^{\mathcal{H}}(\Omega);u=0,\,\,on \,\,\partial\Omega\right\},
\end{align*}
equipped with the norm
\begin{align*}
   \norma[u]{1,\mathcal{H}}{\psi} = \lVert \derivada u\rVert_{\psi}^{\mathcal{H}}+\lVert u\rVert_{\psi}^{\mathcal{H}},
\end{align*}
where $\lVert \derivada u\rVert_{\psi}^{\mathcal{H}}=\lVert \,\lvert\derivada u\rvert\rVert_{\psi}^{\mathcal{H}}$.}

The completion of $\mathrm{C}_{0}^{\infty}(\Omega)$ of $\mathbb{H}_{\mathcal{H},0}^{\alpha,\beta;\psi}(\Omega)$ is denoted by $\mathbb{H}_{\mathcal{H},0}^{\alpha,\beta;\psi}(\Omega)$ and, from \textbf{(H)} (i), we have an equivalent norm on $\mathbb{H}_{\mathcal{H},0}^{\alpha,\beta;\psi}(\Omega)$ given by
\begin{align*}
    \lVert u\rVert_{\psi}^{0,\mathcal{H}} =  \lVert \derivada u\rVert_{\psi}^{\mathcal{H}}.
\end{align*}

The results below will be needed for our purposes.
\begin{proposition}\cite{Wulong}\label{Proposition2.1} Consider that the hypothesis {\rm(i)}  of {\bf (H)}  is  satisfied. Then the following embeddings hold:
\begin{itemize}
    \item [(i)] $L^{\mathcal{H}}(\Omega)\hookrightarrow L^{r}(\Omega)$ and $W^{1,\mathcal{H}}_{0}(\Omega)\hookrightarrow W^{1,r}_{0}(\Omega)$ are continuous for all $r\in [1,p]$;
    
    \item [(ii)] $W^{1,\mathcal{H}}_{0}(\Omega)\hookrightarrow L^{r}(\Omega)$ is continuous for all $r\in[1,p^{*}]$;
    
    \item [(iii)] $W^{1,\mathcal{H}}_{0}(\Omega)\hookrightarrow L^{r}(\Omega)$ is compact for all $r\in [1,p^{*})$;
    
    \item [(iv)] $L^{\mathcal{H}}(\Omega)\hookrightarrow L^{q}_{\mu}(\Omega)$ is continuous;
    
    \item [(vi)] $L^{q}(\Omega)\hookrightarrow L_{\mathcal{H}}(\Omega)$ is continuous.
\end{itemize}

\end{proposition}

\begin{proposition}\cite{Liu}\label{Proposition2.2} Suppose that the hypothesis {\rm(i)}  of {\bf (H)}  is  satisfied and consider $y\in L^{\mathcal{H}}(\Omega)$ and  $\rho^{\mathcal{H}}$  defined in {\rm Eq.(\ref{pl})}. Then the following assertions hold:
\begin{itemize}
\item[(i)] If $y\neq 0$, then $||y||_{\mathcal{H}}=\lambda$ if and only if $\rho^{\mathcal{H}} (y/\lambda)=1$.

\item[(ii)] $||y||_{\mathcal{H}}<1$ $(resp.>1,=1)$ if and only if $\rho^{\mathcal{H}}<1$ $(resp. >1,=1)$.

\item[(iii)] If $||y||_{\mathcal{H}}<1$, then $||y||_{\mathcal{H}}^{q}\leq \rho^{\mathcal{H}}(y)\leq ||y||^{p}_{\mathcal{H}}$;

\item[(iv)] If $||y||_{\mathcal{H}}>1$, then $||y||_{\mathcal{H}}^{p}\leq \rho^{\mathcal{H}}(y)\leq ||y||^{q}_{\mathcal{H}}$;

\item[(v)] $||y||_{\mathcal{H}}\rightarrow 0$ if and only if $\rho^{\mathcal{H}}(y)\rightarrow 0$;

\item[(vi)] $||y||_{\mathcal{H}}\rightarrow +\infty$ if and only if $\rho^{\mathcal{H}}\rightarrow +\infty$.
\end{itemize}
\end{proposition}

\section{Main Results}

Due to the presence of the singular term $a(x)\lvert u\rvert^{1-\gamma}$, we have  that $\mathbf{E}_\lambda$ is not $\mathrm{C}^1$. To overcome this difficulty, we will consider the Nehari manifold method.

Now we consider the fibering function $w_u:[0,+\infty)\rightarrow \mathbb{R}$, for $u\in\mathbb{H}^{\alpha,\beta;\psi}_{\mathcal{H},0}(\Omega) $ defined by
\begin{equation*}
    w_u(t)=\mathbf{E}_\lambda(tu), \, \mbox{ for all } t\geq 0.
\end{equation*}

\textcolor{blue}{The Nehari manifold corresponding to the functional $\mathbf{E}_\lambda$ is defined by
\begin{eqnarray*}
    \mathbf{N}_{\lambda} &=& \left\{u\in \mathbb{H}^{\alpha,\beta;\psi}_{\mathcal{H},0}(\Omega)\setminus \{0\} : \lVert \derivada u\lVert_{L^{p}}^{p} + \lVert \derivada u\lVert_{L_{\mu}^{q}}^{q} = \displaystyle\int_\Omega a(x)|u|^{1-\gamma}\mathrm{d}x +\lambda\lVert u\rVert_{L^{r}}^{r}\right\}\notag\\
    &=& \left\{u\in \mathbb{H}^{\alpha,\beta;\psi}_{\mathcal{H},0}(\Omega)\setminus \{0\} :w'(1) =0 \right\}
\end{eqnarray*}}

\textcolor{blue}{ For further consideration, we need to decompose the set $\mathbf{N}_\lambda$ in the following way:
\begin{eqnarray*}
\mathbf{N}_{\lambda}^+ &=& \left\{u\in  \mathbf{N}_{\lambda} : (p+\gamma-1)\lVert \derivada u\lVert_{L^{p}}^{p} + (q+\gamma-1)\lVert \derivada u\lVert_{L_{\mu}^{q}}^{q} -\lambda(r+\gamma-1)\lVert u\rVert_{L^{r}}^{r}>0\right\}
\notag\\
    &=& \left\{u\in \mathbf{N}_{\lambda} :w''(1) >0  \right\},
\end{eqnarray*}
\begin{eqnarray*}
\mathbf{N}_{\lambda}^0 &=& \left\{ u\in  \mathbf{N}_{\lambda}  : (p+\gamma-1)\lVert \derivada u\lVert_{L^{p}}^{p} + (q+\gamma-1)\lVert \derivada u\lVert_{L_{\mu}^{q}}^{q} = \lambda(r+\gamma-1)\lVert u\rVert_{L^{r}}^{r}\right\}\notag\\
    &=& \left\{ u\in \mathbf{N}_{\lambda} :w''(1)=0 \right\},
\end{eqnarray*}
and
\begin{eqnarray*}
\mathbf{N}_{\lambda}^- &=& \left\{u\in  \mathbf{N}_{\lambda} : (p+\gamma-1)\lVert \derivada u\lVert_{L^{p}}^{p} + (q+\gamma-1)\lVert \derivada u\lVert_{L_{\mu}^{q}}^{q} -\lambda(r+\gamma-1)\lVert u\rVert_{L^{r}}^{r}<0\right\}\notag\\
    &=& \left\{ u\in \mathbf{N}_{\lambda} :w''(1)<0 \right\}.
\end{eqnarray*}}

We start with the following proposition about the coercivity of the energy functional $\mathbf{E}_\lambda$ restricted to $\mathbf{N}_\lambda$.

\begin{proposition}\label{p31}
Suppose that  {\rm\textbf{(H)}} is satisfied. Then $\mathbf{E}_\lambda\lvert_{\mathbf{N}_\lambda}$ is coercive.
\end{proposition}
\begin{proof}
Consider $u\in\mathbf{N}_\lambda$ with $\lVert u\rVert_{\psi}^{\mathcal{H},0}>1$. From the definition of the Nehari manifold $\mathbf{N}_\lambda$ we have
\textcolor{blue}{\begin{equation}\label{eq31}
-\frac{\lambda}{r}\lVert u\rVert_{L^{r}}^{r} = -\frac{1}{r}\left(\lVert \derivada u\lVert_{L^{p}}^{p} + \lVert \derivada u\lVert_{L_{\mu}^{q}}^{q}\right)+\frac{1}{r}\displaystyle\int_\Omega a(x)|u|^{1-\gamma}\mathrm{d}x.
\end{equation}}

Combining (\ref{eq31}) with the definition of $\mathbf{E}_\lambda$ and applying (iv) of  Proposition \ref{Proposition2.2}  along with Theorem 13.17 of Hewitt-Stromberg \cite{Hewitt}, one has
\begin{equation*}
\begin{split}
    \mathbf{E}_\lambda &= \left( \frac{1}{p}-\frac{1}{r} \right)\lVert \derivada u\lVert_{L^{p}}^{p} + \left( \frac{1}{q}-\frac{1}{r} \right)\lVert \derivada u\lVert_{L_{\mu}^{q}}^{q}+\left( \frac{1}{r}-\frac{1}{1-\gamma} \right)\displaystyle\int_\Omega a(x)|u|^{1-\gamma}\mathrm{d}x\\
    &\geq \left( \frac{1}{q}-\frac{1}{r} \right)\rho^{\mathcal{H}}(\derivada u) + \left( \frac{1}{r}-\frac{1}{1-\gamma} \right)\displaystyle\int_\Omega a(x)|u|^{1-\gamma}\mathrm{d}x\\
    & \geq c_1\lVert u\rVert_{\psi}^{\mathcal{H},0,p}-c_2\lVert u\rVert_{\psi}^{\mathcal{H},0,1-\gamma},
\end{split}
\end{equation*}
for some $c_1,c_2>0$, because $p<q<r$. Hence, due to the inequality $1-\gamma<1<p$, the assertion of the proposition follows.
\end{proof}

The result below will be needed.
\begin{proposition}\label{p32}
Suppose that  {\rm\textbf{(H)}}  is satisfied and consider that $\mathbf{N}_{\lambda}^+\neq \emptyset$. If $m_\lambda^+ = \inf_{\mathbf{N}_{\lambda}^+}\mathbf{E}_\lambda$, then $m_{\lambda}^+<0$.
\end{proposition}
\begin{proof}
Let $u\in \mathbf{N}_{\lambda}^+$. First, note that $\mathbf{N}_{\lambda}^+\subseteq \mathbf{N}_{\lambda}$, so we have 
\begin{equation}\label{eq32}
    -\frac{1}{1-\gamma} \displaystyle\int_\Omega a(x)|u|^{1-\gamma}\mathrm{d}x = -\frac{1}{1-\gamma}\left(  \lVert \derivada u\lVert_{L^{p}}^{p} + \lVert \derivada u\lVert_{L_{\mu}^{q}}^{q}\right)+  \frac{\lambda}{1-\gamma}\lVert u\rVert_{L^{r}}^{r}. 
\end{equation}

On the other hand, from the definition of $\mathbf{N}_{\lambda}^+$ yields
\begin{equation}\label{eq33}
     \lambda\lVert u\rVert_{L^{r}}^{r}< \frac{p+\gamma-1}{r+\gamma-1}\lVert \derivada u\lVert_{L^{p}}^{p} + \frac{q+\gamma-1}{r+\gamma-1}\lVert \derivada u\lVert_{L_{\mu}^{q}}^{q}.
\end{equation}

\textcolor{red}{Using Eq.(\ref{eq32}) and Eq.(\ref{eq33}), one has}
\textcolor{blue}{
\begin{align*}
    \mathbf{E}_\lambda (u)= &\frac{1}{p}\lVert \derivada u\lVert_{L^{p}}^{p}+ \frac{1}{q}\lVert \derivada u\lVert_{L_{\mu}^{q}}^{q}-\frac{1}{1-\gamma}\displaystyle\int_\Omega a(x)|u|^{1-\gamma}\mathrm{d}x -\frac{\lambda}{r}\lVert u\rVert_{L^{r}(\Omega)}^{r}\\
     =&\frac{1}{p}\lVert \derivada u\lVert_{L^{p}}^{p}+ \frac{1}{q}\lVert \derivada u\lVert_{L_{\mu}^{q}}^{q} - \frac{1}{1-\gamma}\left(\lVert \derivada u\lVert_{L^{p}}^{p}+ \lVert \derivada u\lVert_{L_{\mu}^{q}}^{q}\right)\\
    & + \frac{\lambda}{1-\gamma}\lVert u\rVert_{L^{r}(\Omega)}^{r}-\frac{\lambda}{r}\lVert u\rVert_{L^{r}(\Omega)}^{r}\\
    =& \left( \frac{1}{p}-\frac{1}{1-\gamma}\right)\lVert\derivada u\rVert_{L^{p}}^{p}+\left( \frac{1}{q}-\frac{1}{1-\gamma}\right)\lVert \derivada u\rVert_{L_{\mu}^{q}}^{q} + \lambda\left(\frac{1}{1-\gamma}-\frac{1}{r}\right) ||u||^{r}_{L^{r}(\Omega)}\\
    \leq& \left[ \frac{1-\gamma-p}{p(1-\gamma)}+\left(\frac{-1+\gamma+r}{r(1-\gamma)}\right)\frac{p+\gamma-1}{r+\gamma-1} \right]\lVert \derivada u\lVert_{L^{p}}^{p}\\ &+\left[ \frac{1-\gamma-q}{q(1-\gamma)}+\left(\frac{-1+\gamma+r}{r(1-\gamma)}\right)\frac{q+\gamma-1}{r+\gamma+1} \right]\lVert \derivada u\lVert_{L_{\mu}^{q}}^{q}\\
    =& \frac{p+\gamma-1}{1-\gamma}\left( \frac{1}{r}-\frac{1}{p}\right)\lVert \derivada u\lVert_{L^{p}}^{p}+  \frac{q+\gamma-1}{1-\gamma}\left( \frac{1}{r}-\frac{1}{q}\right)\lVert \derivada u\lVert_{L_{\mu}^{q}}^{q},\\
    <& 0,
\end{align*}}
since $p<q<r$. Hence, $\mathbf{E}_\lambda\lvert_{\mathbf{N}_{\lambda}^+}<0$, and so $m_{\lambda}^+<0$.
\end{proof}

\begin{proposition}\label{p33}
Suppose that  {\rm\textbf{(H)}} is satisfied. Then there exists $\lambda^*>0$ such that ${\bf N}_\lambda^0=\emptyset$ for all $\lambda \in (0,\lambda^*)$.
\end{proposition}
\begin{proof}
Suppose that for every $\lambda^*>0$ there exists $\lambda \in (0,\lambda^*)$ such that ${\bf N}_\lambda^0\neq\emptyset$. Hence, for any given $\lambda >0$, we can find $u \in {\bf N}_\lambda^0 $ such that
\begin{equation}\label{eq34}
  (p+\gamma-1)  \lVert \derivada u\lVert_{L^{p}}^{p} + (q+\gamma-1)\lVert \derivada u\lVert_{L_{\mu}^{q}}^{q} = \lambda(r+\gamma-1)\lVert u\rVert_{L^{r}}^{r}
\end{equation}
Since $u\in {\bf N}_\lambda$, one also has
\begin{eqnarray}\label{eq35}
      (r+\gamma-1)  \lVert \derivada u\lVert_{L^{p}}^{p} + (r+\gamma-1)\lVert \derivada u\lVert_{L_{\mu}^{q}}^{q} &=  &(r+\gamma-1)\displaystyle\int_\Omega a(x)|u|^{1-\gamma}\mathrm{d}x\nonumber\\&+& \lambda(r+\gamma-1)\lVert u\rVert_{L^{r}}^{r}.
\end{eqnarray}

\textcolor{red}{ Using Eq.(\ref{eq34}) and Eq.(\ref{eq35}), yields}
\begin{equation}\label{eq36}
    (r-p)\lVert \derivada u\lVert_{L^{p}}^{p} + (r-q)\lVert \derivada u\lVert_{L_{\mu}^{q}}^{q} = (r+\gamma -1)\displaystyle\int_\Omega a(x)|u|^{1-\gamma}\mathrm{d}x.
\end{equation}

From (ii) and (iv) of Proposition \ref{Proposition2.2}  along with Theorem 13.17 of Hewitt - Stromberg and (ii) of Proposition \ref{Proposition2.1} , we get from Eq.(\ref{eq36}), that
\begin{equation*}
    \min\left\{\lVert u \rVert_{\mathbb{H}^{\alpha,\beta;\psi}_{\mathcal{H},0}}^{p},\lVert u \rVert_{\mathbb{H}^{\alpha,\beta;\psi}_{\mathcal{H},0}}^{q}\right\}\leq c_3\lVert u \rVert_{\mathbb{H}^{\alpha,\beta;\psi}_{\mathcal{H},0}}^{1-\gamma},
\end{equation*}
for some $c_3>0$ due to the fact that $1-\gamma<1<p<q<r$. Hence,
\begin{equation}\label{eq37}
    \lVert u \rVert_{\mathbb{H}^{\alpha,\beta;\psi}_{\mathcal{H},0}}\leq c_4,
\end{equation}
for some $c_4>0$.

Furthermore, from Eq.(\ref{eq34}),  Proposition \ref{Proposition2.2} (ii), (iv) and Proposition \ref{Proposition2.1} (ii), we have 
 \begin{equation*}
    \min\left\{\lVert u \rVert_{\mathbb{H}^{\alpha,\beta;\psi}_{\mathcal{H},0}}^{p},\lVert u \rVert_{\mathbb{H}^{\alpha,\beta;\psi}_{\mathcal{H},0}}^{q}\right\}\leq \lambda c_5\lVert u \rVert_{\mathbb{H}^{\alpha,\beta;\psi}_{\mathcal{H},0}}^{r},
\end{equation*}
for some $c_5>0$. Consequently,
\begin{equation*}
     \lVert u \rVert_{\mathbb{H}^{\alpha,\beta;\psi}_{\mathcal{H},0}}\geq \left(\frac{1}{\lambda c_5}\right)^{\frac{1}{r-p}} \mbox{ or }   \lVert u \rVert_{\mathbb{H}^{\alpha,\beta;\psi}_{\mathcal{H},0}}\geq \left(\frac{1}{\lambda c_5}\right)^{\frac{1}{r-q}}.
\end{equation*}
If $\lambda \rightarrow 0^+$,  then $\lVert u \rVert_{\mathbb{H}^{\alpha,\beta;\psi}_{\mathcal{H},0}} \rightarrow +\infty$, due to the inequality $p<q<r$, which contradicts Eq.(\ref{eq37}).
\end{proof}

\begin{proposition}\label{p34}
Suppose that  \rm\textbf{(H)} is  satisfied. Then there exists $\hat{\lambda}^* \in (0,\lambda^*]$ such that $\mathbf{N}_\lambda^{\pm}\neq \emptyset$ for all $\lambda \in (0,\hat{\lambda}^*)$. In addition, for any $\lambda \in (0,\hat{\lambda}^*)$, there exists $u^*\in \mathbf{N}_\lambda^{*}$ such that $\mathbf{E}_\lambda(u^*)=m_{\lambda}^{+}<0$ and $u^*(x)\geq 0$ for a.a. $x\in \Omega$.
\end{proposition}
\begin{proof}
Let $u\in\hzero(\Omega)\setminus \{0\}$ and consider the function $\hat{\Phi}_u:(0,+\infty) \rightarrow \mathbb{R}$ defined by 
\begin{equation*}
    \hat{\Phi}_u(t) = t^{p-r}\lVert\derivada u \rVert_{L^{p}(\Omega)}^{p} - t^{-r-\gamma+1}\displaystyle\int_\Omega a(x)|u|^{1-\gamma}\mathrm{d}x.
\end{equation*}
\end{proof}

Since $r-p<r+\gamma-1$, we can find $\hat{t}_0>0$ such that $  \hat{\Phi}_u(\hat{t}_0) = \max_{t>0} \hat{\Phi}_u(t)$. 
Thus, $\hat{\Phi}'_u(\hat{t}_0)=0$, i.e.
\begin{equation*}
    (p-r)\hat{t}_0^{p-r-1}\lVert \derivada u \rVert_{L^{p}(\Omega)}^{p}+(r+\gamma-1)\hat{t}_0^{-r-\gamma}\displaystyle\int_\Omega a(x)|u|^{1-\gamma}\mathrm{d}x = 0.
\end{equation*}

\textcolor{blue}{Hence
\begin{equation*}
    \hat{t}_0 = \left(\frac{r+\gamma-1)\displaystyle\int_\Omega a(x)|u|^{1-\gamma}\mathrm{d}x}{(r-p)\lVert \derivada u \rVert_{L^{p}(\Omega)}^{p}}\right)^\frac{1}{p+\gamma-1}.
\end{equation*}
Furthermore, we have
\begin{equation}\label{eq38}
    \begin{split}
        \hat{\Phi}_u(\hat{t}_0) = &\left(\frac{(r-p)\lVert \derivada u \rVert_{L^{p}(\Omega)}^{p}}{(r+\gamma-1)\displaystyle\int_\Omega a(x)|u|^{1-\gamma}\mathrm{d}x}\right)^{\frac{r-p}{p+\gamma-1}}\lVert \derivada u \rVert_{L^{p}(\Omega)}^{p}\\
        & - \left(\frac{(r-p)\lVert \derivada u \rVert_{L^{p}(\Omega)}^{p}}{(r+\gamma-1)\displaystyle\int_\Omega a(x)|u|^{1-\gamma}\mathrm{d}x}\right)^{\frac{r+\gamma-1}{p+\gamma-1}}\displaystyle\int_\Omega a(x)|u|^{1-\gamma}\mathrm{d}x\\
        = & \frac{ (r-p)^\frac{r-p}{p+\gamma-1}\lVert \derivada u \rVert_{L^{p}(\Omega)}^{\frac{p(r+\gamma-1)}{p+\gamma-1}}}{(r+\gamma-1)^\frac{r-p}{p+\gamma-1}\left(\displaystyle\int_\Omega a(x)|u|^{1-\gamma}\mathrm{d}x\right)^{\frac{r-p}{p+\gamma-1}}} \\
        & - \frac{ (r-p)^\frac{r+\gamma-1}{p+\gamma-1}\lVert \derivada u \rVert_{L^{p}(\Omega)}^{\frac{p(r+\gamma-1)}{p+\gamma-1}}}{(r+\gamma-1)^\frac{r+\gamma-1}{p+\gamma-1}\left(\displaystyle\int_\Omega a(x)|u|^{1-\gamma}\mathrm{d}x\right)^{\frac{r-p}{p+\gamma-1}}}\\
        = & \frac{p+\gamma-1}{r-p}\left(\frac{r-p}{r+\gamma-1}\right)^{\frac{r+\gamma-1}{p+\gamma-1}}\frac{\lVert \derivada u \rVert_{L^{p}(\Omega)}^{\frac{p(r+\gamma-1)}{p+\gamma-1}}}{\left(\displaystyle\int_\Omega a(x)|u|^{1-\gamma}\mathrm{d}x\right)^{\frac{r-p}{p+\gamma-1}}}.
    \end{split}
\end{equation}}

Let $S$ be the best constant of the Sobolev embedding $\mathbb{H}_p^{\alpha, \beta;\psi}(\Omega)\rightarrow L^{p*}(\Omega)$, that is
\begin{equation}\label{eq39}
    S\lVert u\rVert_{L^{p*}(\Omega)}^p\leq \lVert \derivada \rVert_{L^{p}(\Omega)}^p.
\end{equation}

Moreover, we have
\begin{equation}\label{eq310}
    \displaystyle\int_\Omega a(x)|u|^{1-\gamma}\mathrm{d}x\leq c_6\lVert u\rVert_{L^{p*}(\Omega)}^{1-\gamma},
\end{equation}
for some $c_6>0$. From Eq.(\ref{eq38}), Eq.(\ref{eq39}) and Eq.(\ref{eq310}), we obtain
\begin{eqnarray*}
     \hat{\Phi}_u(\hat{t}_0) - \lambda\lVert u\rVert_{L^{r}(\Omega)}^r &=&   \frac{p+\gamma-1}{r-p}\left(\frac{r-p}{r+\gamma-1}\right)^{\frac{r+\gamma-1}{p+\gamma-1}}\frac{\lVert \derivada u \rVert_{L^{p}(\Omega)}^{\frac{p(r+\gamma-1)}{p+\gamma-1}}}{\left(\displaystyle\int_\Omega a(x)|u|^{1-\gamma}\mathrm{d}x\right)^{\frac{r-p}{p+\gamma-1}}}- \lambda\lVert u\rVert_{L^{r}(\Omega)}^r\\
     &\geq& \frac{p+\gamma-1}{r-p}\left(\frac{r-p}{r+\gamma-1}\right)^{\frac{r+\gamma-1}{p+\gamma-1}}\frac{S^{\frac{r+\gamma-1}{p+\gamma-1}}\left(\lVert u\rVert _{L^{p*}(\Omega)}^p\right)^{\frac{r+\gamma-1}{p+\gamma-1}}}{\left(c_6\lVert u\rVert _{L^{p*}(\Omega)}^{1-\gamma}\right)^{\frac{r-p}{p+\gamma-1}}}-\lambda c_7\lVert u\rVert _{L^{p*}(\Omega)}^{r}\notag\\&=& [c_{8}-c_{7}] ||u||^{r}_{L^{p^{*}}(\Omega)},
\end{eqnarray*}
for some $c_7,c_{8}>0$. Therefore, there exists $\hat{\lambda}^*\in (0, \lambda^*]$, independent of $u$, such that
\begin{equation}\label{eq311}
     \hat{\Phi}_u(\hat{t}_0) - \lambda\lVert u\rVert_{L^{r}(\Omega)}^r>0, \mbox{ for all } \lambda\in (0,\hat{\lambda}^*).
\end{equation}

Now consider the function $\Phi_u: (0,+\infty)\rightarrow \mathbb{R}$, defined by
\begin{equation*}
    \Phi_u (t)= t^{p-r}\lVert \derivada \rVert_{L^{p}(\Omega)}^p + t^{q-r}\lVert \derivada \rVert_{L_{\mu}^{q}(\Omega)}^q - t^{-r-\gamma+1}\displaystyle\int_\Omega a(x)|u|^{1-\gamma}\mathrm{d}x.
\end{equation*}
Since $r-q<r-p<r+\gamma-1$, we can find $t_0>0$ such that 
\begin{equation*}
    \Phi_u(t_0)=\max_{t>0} \Phi_u(t).
\end{equation*}
Because of $\Phi_u\geq \hat{\Phi}_u$, and due to Eq.(\ref{eq311}), we can find $\hat{\lambda}^*\in (0,\lambda^*]$ independent of $u$ such that
\begin{equation*}
    \Phi_u(t_0)-\lambda\norma[u]{r}{L^r(\Omega)}>0 \, \mbox{ for all } \lambda \in (0, \hat{\lambda}^*).
\end{equation*}
Thus, there exist $t_1<t_0<t_2$ such that $\Phi_u(t_1)=\lambda\norma[u]{r}{L^r(\Omega)}= \Phi_u(t_2) $ and 
\begin{equation}\label{eq312}
    \Phi'_u(t_2)<0<\Phi'_u(t_1),
\end{equation}
where 

\textcolor{blue}{\begin{equation}\label{eq313}
\begin{split}
      \Phi'_u(t)=(p-r)t^{p-r-1}\norma[\derivada u]{p}{L^p(\Omega)}+ (q-r)t^{q-r-1}\norma[\derivada u]{q}{L^q(\Omega)} \\- (-r-\gamma+1)t^{-r-\gamma}\displaystyle\int_\Omega a(x)|u|^{1-\gamma}\mathrm{d}x.
\end{split}
\end{equation}}

Recall that the the fibering function $w_u:[0,+\infty)\rightarrow \mathbb{R}$, defined by
\begin{equation*}
    w_u(t)=\mathbf{E}_\lambda(t_u), \, \mbox{ for all } t\geq 0.
\end{equation*}
First, we see that $w_u\in \mathrm{C}^2(0,\infty)$. I this sense, we get
\begin{equation*}
    w'_u(t_1)=t_1^{p-1}\norma[\derivada u]{p}{L^p(\Omega)} + t_1^{q-1}\norma[\derivada u]{q}{L^q(\Omega)} - t_1^{-\gamma}\displaystyle\int_\Omega a(x)|u|^{1-\gamma}\mathrm{d}x - \lambda t_1^{r-1}\norma[u]{r}{L^r(\Omega)}
\end{equation*}
and
\textcolor{blue}{\begin{eqnarray}\label{eq314}
w''_u(t_1)&=& (p-1)t_1^{p-2}\norma[\derivada u]{p}{L^p(\Omega)}+ (q-1)t_1^{q-2}\norma[\derivada u]{q}{L^q(\Omega)}\notag\\ 
      &+&\gamma t_1^{-\gamma-1}\displaystyle\int_\Omega a(x)|u|^{1-\gamma}\mathrm{d}x- \lambda(r-1)t_1^{r-2}\norma[u]{r}{L^r(\Omega)}.
\end{eqnarray}}

From Eq.(\ref{eq312}), we get
\begin{equation*}
    t_1^{p-r}\norma[\derivada u]{p}{L^p(\Omega)}+t_1^{q-r}\norma[\derivada u]{q}{L^q(\Omega)}- t_1^{-r-\gamma+1}\displaystyle\int_\Omega a(x)|u|^{1-\gamma}\mathrm{d}x = \lambda \norma[u]{r}{L^r(\Omega)}.
\end{equation*}

This implies, multiplying by $\gamma t_1^{r-2}$ and $-(r-1)t_1^{r-2}$, respectively, that
\begin{equation}\label{eq315}
 \gamma t_1^{p-2}\norma[\derivada u]{p}{L^p(\Omega)} + \gamma t_1^{q-2}\norma[\derivada u]{q}{L^q(\Omega)} - \gamma\lambda t_1^{r-2}\lambda \norma[u]{r}{L^r(\Omega)} \\= \gamma t_1^{-\gamma -1}\displaystyle\int_\Omega a(x)|u|^{1-\gamma}\mathrm{d}x
\end{equation}
and
\begin{equation}\label{eq316}
\begin{split}
    (r-1)t_1^{-\gamma-1}\displaystyle\int_\Omega a(x)|u|^{1-\gamma}\mathrm{d}x -(r-1)t_1^{p-2}\norma[\derivada u]{p}{L^p(\Omega)} - (r-1)t_1^{q-2}\norma[\derivada u]{q}{L^q(\Omega)} \\= -\lambda(r-1)t_1^{r-2}\norma[u]{r}{L^r(\Omega)}. 
\end{split}
\end{equation}

Replacing the Eq.(\ref{eq315}) in Eq.(\ref{eq314}), yields
\begin{equation}\label{eq317}
    \begin{split}
        w''_u(t_1)= & (p+\gamma -1)t_1^{p-2}\norma[\derivada u]{p}{L^p(\Omega)} + (q+\gamma-1)t_1^{q-2}\norma[\derivada u]{q}{L^q(\Omega)}- \lambda (r+\gamma-1)t_1^{r-2}\norma[u]{r}{L^r(\Omega)}\\
         = & t_1^{-2}\left( (p+\gamma-1)t_1^p\norma[\derivada u]{p}{L^p(\Omega)} + (q+\gamma-1)t_1^q \norma[\derivada u]{q}{L^q(\Omega)}\right.\left.- \lambda(r+\gamma-1)t_1^r\norma[u]{r}{L^r(\Omega)}\right).
    \end{split}
\end{equation}

On the other hand, applying the Eq.(\ref{eq316}) in Eq.(\ref{eq314}) and using the representation in Eq.(\ref{eq313}) it follows that
\textcolor{blue}
{\begin{equation}\label{eq318}
\begin{split}
     w''_u(t_1) =& (p-r)t_1^{p-2}\norma[\derivada u]{p}{L^p(\Omega)} + (q-r)t_1^{q-2}\norma[\derivada u]{q}{L^q(\Omega)} \\&+ (r+\gamma-1)t_1^{-\gamma-1}\displaystyle\int_\Omega a(x)|u|^{1-\gamma}\mathrm{d}x\\
     = & t_1^{1-r}\phi'_u(t_1) >0.
\end{split}
\end{equation}}

\textcolor{red}{ From Eq.(\ref{eq317}) and Eq.(\ref{eq318}), we obtain}
\begin{equation*}
    (p+\gamma-1)t_1^p\norma[\derivada u]{p}{L^p(\Omega)} + (q+\gamma-1)t_1^q\norma[\derivada u]{q}{L^q(\Omega)} - \lambda (r+\gamma-1)t_1^r\norma[u]{r}{L^r(\Omega)}>0,
\end{equation*}
which implies 
\begin{equation*}
    t_1u\in\mathbf{N}_\lambda^+ \, \mbox{ for all } \lambda \in (0, \hat{\lambda}^*].
\end{equation*}
Hence, $\mathbf{N}_\lambda^+\neq \emptyset$.
 
 Similarly, it is shown that $\mathbf{N}_\lambda^-\neq \emptyset$ for $t_2$.
 
 Let $\{u_n\}_{n\in\N}\subset \mathbf{N}_\lambda^+$ be a minimizing sequence; i.e.,
 \begin{equation}\label{eq319}
     \mathbf{E}_\lambda(u_n)\searrow m_\lambda^+<0 \, \mbox{ as } n\rightarrow \infty.
 \end{equation}
Recall that $\mathbf{N}_\lambda^+\subset \mathbf{N}_\lambda$ and so we conclude from  Proposition \ref{p31} that $\{u_n\}_{n\in\N}\subset \hzero(\Omega)$ is bounded. Therefore, we may assume that
\begin{equation}\label{eq320}
    u_n\rightarrow u^* \, \mbox{ in } \hzero(\Omega) \, \mbox{ and } u_n\rightarrow u^* \, \mbox{ in } L^{r}(\Omega).
\end{equation}

\textcolor{blue}{From Eq.(\ref{eq319}) and Eq.(\ref{eq320}), we know that
\begin{equation*}
    \mathbf{E}_\lambda(u^*) \leq \liminf_{n\rightarrow +\infty}\mathbf{E}_\lambda(u_n)<0={\bf E}_{\lambda}(0).
\end{equation*}}
Consequently, $u^*\neq 0$.

\textbf{\underline{Claim}:} $\displaystyle\liminf_{n\rightarrow +\infty}\rho^{\mathcal{H}}(u_n)= \rho^{\mathcal{H}}(u^*)$.

Suppose, by contradiction, that
\begin{equation*}
  \liminf_{n\rightarrow +\infty}\rho^{\mathcal{H}}(u_n)> \rho^{\mathcal{H}}(u^*).
\end{equation*}
From Eq.(\ref{eq312}), yields
\begin{equation*}
    \begin{split}
        \liminf_{n\rightarrow +\infty}w'_{u_n}(t_1) = &  \liminf_{n\rightarrow +\infty}\left( t_1^{p-1}\norma[\derivada u]{p}{L^p(\Omega)} + t_1^{q-1}\norma[\derivada u]{q}{L^q(\Omega)} \right.\\&\left.- t_1^{-\gamma}\displaystyle\int_\Omega a(x)|u|^{1-\gamma}\mathrm{d}x - \lambda t_1^{r-1}\norma[u]{r}{L^r(\Omega)}\right).\\
        >& t_1^{p-1}\norma[\derivada u^*]{p}{L^{p}(\Omega)} + t_1^{q-1}\norma[\derivada u^*]{q}{L_{\mu}^q} - t_1^{-\gamma}\displaystyle\int_\Omega a(x)|u^*|^{1-\gamma}\mathrm{d}x - \lambda t_1^{r-1}\norma[u^*]{r}{L^r(\Omega)}\\
        = & w'_{u^*}(t_1)\\
        = & t_1^{r-1}\left(\Phi_{u^*}(t_1)-\lambda \norma[u^*]{r}{L^{r}(\Omega)}\right)\\
        = &0,
    \end{split}
\end{equation*}
which implies the existence of $n\in\N$ such that $w'_{u_n}(t_1)>0$, for all $n>n_0$. Recall that $u_n\in \mathbf{N}_\lambda^+\subset \mathbf{N}_\lambda$ and $w'_{u_n}(t) = t^{r-1}(\Phi_{u_n}(t)-\lambda \norma[u_n]{r}{L^{r}(\Omega)})$. Thus, we have $w'_{u_n}(t)<0$, for all $t\in (0,1)$ and $w'_{u_n}(1)=0$. Therefore, $t_1>1$.

Since $w_{u^*}$ is decreasing on $(0,t_1]$, one has
\begin{equation*}
    \mathbf{E}_\lambda(t_1u^*)\leq  \mathbf{E}_\lambda(u^*) < m_\lambda^+.
\end{equation*}

Remember that $t_1u^*\in \mathbf{N}_\lambda^+$. So, we obtain that
\begin{equation*}
    m_\lambda^+\leq \mathbf{E}_\lambda(t_1u^*) < m_\lambda^+,
\end{equation*}
which is a contradiction, and the claim is proved.

Since $\displaystyle\liminf_{n\rightarrow +\infty}\rho^{\mathcal{H}}(u_n)= \rho^{\mathcal{H}}(u^*)$, we can find a subsequence, still denoted by $u_n$, such that $\rho^\mathcal{H}(u_n)\rightarrow \rho^\mathcal{H}(u^*)$. It follows from (v) of Proposition \ref{Proposition2.2}  that $u_n\rightarrow u$ in $\hzero(\Omega)$. This implies $\mathbf{E}_\lambda(u_n)\rightarrow \mathbf{E}_\lambda(u^*)$ and, consequently, $\mathbf{E}_\lambda(u^*)=m_\lambda^+$. Since $u_n\in \mathbf{N}_\lambda^+$ for all $n\in\N$, it follows that
\begin{equation*}
    (p+\gamma-1)\norma[\derivada u_n]{p}{L^{p}} + (q+\gamma -1)\norma[\derivada u_n]{q}{L_{\mu, \psi}^{q}} - \lambda(r+\gamma-1)\norma[u_n]{r}{L^r(\Omega)}>0.
\end{equation*}

\textcolor{blue}{Taking $n\rightarrow +\infty$, we obtain
\begin{equation}\label{eq321}
 (p+\gamma-1)\norma[\derivada u^*]{p}{L^{p}} + (q+\gamma -1)\norma[\derivada u^*]{q}{L_{\mu, \psi}^{q}} - \lambda(r+\gamma-1)\norma[u^*]{r}{L^r(\Omega)}\geq 0.
\end{equation}}

Recall that $\lambda\in (0, \hat{\lambda}^*)$ and $\hat{\lambda}^*\leq \lambda^*$. Then, from Proposition \ref{p33}, we know that equality in Eq.(\ref{eq321}) can not occur. Therefore, we conclude that $u^*\in \mathbf{N}_\lambda^+ $. Since we can use $|u^*|$ instead of $u^*$, we may assume that $u^*(x)\geq 0$ for a.a $x\in \Omega$ with $u^*\neq 0$. The proof is finished.

In what follows, for $\epsilon>0$, we denote
\begin{equation*}
    B_\epsilon(0)=\left\{u\in\hzero(\Omega): \norma[u]{}{\hzero(\Omega)}<\epsilon\right\}.
\end{equation*}
\begin{lemma}\label{l35}
Suppose that {\rm(\bf H)} is satisfied and let $u\in \mathbf{N}_\lambda^{\pm}$. Then there exist $\epsilon>0$ and a continuous $v: B_\epsilon(0)\rightarrow (0,+\infty) $ such that
\begin{equation*}
    v(0)=1 \,\mbox{ and }\, v(y)(u+y)\in \mathbf{N}_\lambda^{\pm},\, \mbox{ for all } y\in B_\epsilon(0).
\end{equation*}
\end{lemma}

\begin{proof} We show the proof only for $\mathbf{N}_\lambda^{+}$, since the proof for $\mathbf{N}_\lambda^{-}$ works in a similar way. To this end, let $\overline{\delta}:\hzero(\Omega)\times (0,+\infty)\rightarrow\mathbb{R}$ be defined by
\begin{align*}
    \overline{\delta}(y,t) =& t^{p+\gamma-1} \norma[\derivada(u+y)]{p}{L^{p}} + t^{q+\gamma-1} \norma[\derivada(u+y)]{q}{L_{\mu}^{q}} - \displaystyle\int_\Omega a(x)|u+y|^{1-\gamma}\mathrm{d}x\\
    & - \lambda t_1^{r+\gamma-1}\norma[u+y]{r}{L^r(\Omega)},
\end{align*}
for all $y\in\hzero(\Omega)$.

Since $u\in\mathbf{N}_\lambda^{+}\subset\mathbf{N}_\lambda$, one has $\overline{\delta}(0,1)=0$. Because of $u\in \mathbf{N}_\lambda^{+}$, it follows that
\begin{equation*}
    \overline{\delta}'_t(0,1) = (p+\gamma-1)\norma[\derivada u]{p}{L^{p}(\Omega)} + (q+\gamma-1)\norma[\derivada u]{q}{L_{\mu}^{q}(\Omega)} - \lambda(r+\gamma-1)\norma[u]{r}{L^r(\Omega)}>0.
\end{equation*}

Then, by the  Implicit Function Theorem, there exist $\epsilon>0$ and a continuous function $\overline{\delta}:B_\epsilon(0)\rightarrow (0,+\infty)$ such that
\begin{equation*}
    \overline{\delta}(0)=1 \, \mbox{ and } \overline{\delta}(y)(u+y)\in \mathbf{N}_\lambda^{+}, \mbox{ for all } y \in B_\epsilon(0).
\end{equation*}
\end{proof}

\begin{proposition}\label{p36} Suppose that  {\rm(\bf H)} is satisfied and consider $h\in \hzero(\Omega)$ and $\lambda \in(0, \hat{\lambda}]$. Then there exists $b>0$ such that $\mathbf{E}_\lambda(u^*)\leq \mathbf{E}_\lambda(u^*+th)$, for all $t\in[0,b]$.
\end{proposition}

\begin{proof} We introduce the function $\eta_h:[0,+\infty)\rightarrow\mathbb{R}$, defined by
\begin{equation}\label{eq322}
\begin{split}
  \eta_h(t) = (p-1)\norma[\derivada u^*+t\derivada h]{p}{L^{p}(\Omega)} + (q-1)\norma[\derivada u^*+t\derivada h]{q}{L_{\mu}^{q}(\Omega)} \\
  + \gamma  \displaystyle\int_\Omega a(x)|u^*+th|^{1-\gamma}\mathrm{d}x - \lambda(r-1)\norma[u^*+th]{r}{L^r(\Omega)}.
\end{split}
\end{equation}

As we can see in  Proposition \ref{p34}, $u^*\in \mathbf{N}_\lambda^{+}\subseteq	 \mathbf{N}_\lambda$. This implies that
\begin{equation}\label{eq323}
\begin{split}
    \gamma \displaystyle\int_\Omega a(x)|u^*|^{1-\gamma}\mathrm{d}x =  \gamma \norma[\derivada u^*]{p}{L^{p}(\Omega)} +\gamma \norma[\derivada u^*]{q}{L_{\mu}^{q}(\Omega)} -\lambda\gamma\norma[u^*]{r}{L^r(\Omega)} 
\end{split}
\end{equation}
and
\begin{equation}\label{eq324}
\begin{split}
    0<(p+\gamma-1)\norma[\derivada u^*]{p}{L^{p}(\Omega)} + (q+\gamma-1)\norma[\derivada u^*]{q}{L_{\mu}^{q}(\Omega)}- \lambda(r+\gamma-1)\norma[u^*]{r}{L^r(\Omega)}. 
    \end{split}
\end{equation}

Combining Eq.(\ref{eq322}), Eq.(\ref{eq323}) and Eq.(\ref{eq324}), we conclude that $\eta_h(0)>0$. Since $\eta_h:[0,+\infty)\rightarrow\mathbb{R}$ is continuous, we can find $b_0>0$ such that
\begin{equation*}
    \eta_h(t)>0 \mbox{ for all } t \in [0,b_0].
 \end{equation*}

Using Lemma \ref{l35}, this implies that for every $t\in[0,b_0]$, we can find $\overline{\delta}(t)>0$ such that
 \begin{equation}\label{eq325}
     \overline{\delta}(t)(u^*+th)\in  \mathbf{N}_\lambda^{+} \mbox{ and } \overline{\delta}(t)\rightarrow 1 \mbox{ as } t\rightarrow 0^+.
 \end{equation}

\textcolor{blue}{From Proposition \ref{p34}, we know that
 \begin{equation*}
     m_{\lambda}^+ = {\bf E}_\lambda(u^*)\leq {\bf E}_\lambda(\overline{\delta}(u^*+th)) \mbox{ for all } t\in[0,b_0].
 \end{equation*}}

 From $w''_{{u}^{*}}(1)>0$ and the continuity in $t$, we get $w''_{{u}^{*}+th}(1)>0$ for $t\in [0,b]$, with $b\in (0,b_{0}]$. Therefore, from this and Eq.(\ref{eq325}), we can find $b\in(0,b_0]$ small enough such that
 \begin{equation*}
     m_{\lambda}^+ = {\bf E}_\lambda(u^*)\leq {\bf E}_\lambda(u^*+th) \mbox{ for all } t\in[0,b).
 \end{equation*}
\end{proof}

\begin{proposition}\label{p37} Suppose that {\rm(\bf H)} is satisfied and let $\lambda\in(0,\hat{\lambda}^*]$. Then $u^*$ is a weak solution of problem {\rm Eq.(\ref{eq1})} such that $\mathbf{E}_\lambda(u^*)<0$.
\end{proposition}

\begin{proof} Let $0\leq \overline{h}\in\hzero(\Omega)$. Through the Proposition \ref{p36}, we know that 
\begin{equation*}
    0\leq \mathbf{E}_\lambda(u^*+t\overline{h})-\mathbf{E}_\lambda(u^*)\, \mbox{ for all }0\leq t\leq b,
\end{equation*}
that is,
\begin{eqnarray}\label{eq326}
\frac{1}{1-\gamma}\displaystyle\int_\Omega a(x)\left[|u^*+t\overline{h}|^{1-\gamma}-|u^*|^{1-\gamma}\right]\mathrm{d}x \notag &&\leq  \displaystyle\int_\Omega \frac{1}{p}\left[|\derivada u^* +t\derivada \overline{h}|^p - |\derivada u^*|^p\right]\mathrm{d}x\notag\\
+&& \displaystyle\int_\Omega \frac{\mu(x)}{q}\left[|\derivada u^*+t\derivada \overline{h}|^q - |\derivada u^*|^q\right]\mathrm{d}x\notag\\
-&& \dfrac{\lambda}{r}\left[\norma[u^*+t\overline{h}]{r}{L^r(\Omega)} - \norma[u^*]{r}{L^r(\Omega)}\right]
\end{eqnarray}

Multiplying (\ref{eq326}) by $\dfrac{1}{t}$ and taking the limit $t\rightarrow 0^+$, we have
\begin{equation}\label{eq327}
\begin{split}
    \frac{1}{1-\gamma}\lim_{t\rightarrow 0^+}\displaystyle\int_\Omega\left(\frac{ a(x)\left[|u^*+t\overline{h}|^{1-\gamma}-|u^*|^{1-\gamma}\right]}{t}\right)\mathrm{d}x\leq \displaystyle\int_\Omega\left(|\derivada u^*|^{p-2}\derivada u^*\right)\derivada\overline{h}\,\mathrm{d}x\\+\displaystyle\int_\Omega\left( \mu(x)|\derivada u^*|^{q-2}\derivada u^*  \right)\derivada\overline{h}\,\mathrm{d}x\\-\lambda \displaystyle\int_\Omega (u^*)^{r-1}\overline{h}\,\mathrm{d}x.
\end{split}
\end{equation}
Note that
\begin{equation*}
     \frac{1}{1-\gamma}\displaystyle\int_\Omega\frac{ a(x)\left[|u^*+t\overline{h}|^{1-\gamma}-|u^*|^{1-\gamma}\right]}{t}\mathrm{d}x=\displaystyle\int_\Omega a(x)(u^*+\epsilon t\overline{h})^{-\gamma}\overline{h} \mathrm{d}x,
\end{equation*}
where $\epsilon\rightarrow 0^+$ as $t\rightarrow 0^+$ and $a(x)(u^*+\epsilon t\overline{h})^{-\gamma}\overline{h}\rightarrow a(x) (u^*)^{-\gamma}\overline{h}$ a.e. in $\Omega$, as $t\rightarrow 0^+$.

We point out that $a(x)(u^*+\epsilon t\overline{h})^{-\gamma}\overline{h}\geq 0$ in $\Omega$. It follows from Fatou's Lemma that $a(x) (u^*)^{-\gamma}\overline{h}$ is integrable and
\begin{equation}\label{eq328}
    \displaystyle\int_\Omega a(x) (u^*)^{-\gamma}\overline{h} \mathrm{d}x\leq  \frac{1}{1-\gamma}\lim_{t\rightarrow 0^+}\displaystyle\int_\Omega\left(\frac{ a(x)\left[|u^*+t\overline{h}|^{1-\gamma}-|u^*|^{1-\gamma}\right]}{t}\right)\mathrm{d}x.
\end{equation}

Putting Eq.(\ref{eq327}) and Eq.(\ref{eq328}) together, we can find that 
\begin{equation}\label{eq329}
\begin{split}
    0\leq& \displaystyle\int_\Omega\left(|\derivada u^*|^{p-2}\derivada u^*+ \mu(x)|\derivada u^*|^{q-2}\derivada u^*  \right)\derivada\overline{h}\,\mathrm{d}x\\
    &-\displaystyle\int_\Omega  \left(a(x) (u^*)^{-\gamma}+\lambda(u^*)^{r-1} \right)\overline{h}\,\mathrm{d}x,
\end{split}
\end{equation}
for all $\overline{h}\in\hzero(\Omega)$ with $\overline{h}\geq 0$.

Now we are ready to prove that $u^*$ is a weak solution of Eq.(\ref{eq1}). To this end, suppose that $h\in \hzero(\Omega)$ and $\epsilon>0$, and define $\Theta\in \hzero(\Omega)$, $\Theta\geq 0$ by
\begin{equation*}
    \Theta = (u^*+\epsilon h)^+ = \max (u^*+\epsilon h,0).
\end{equation*}

\textcolor{blue}{Replacing $\overline{h}$ by $\Theta$ in Eq.(\ref{eq329}), we obtain
\begin{equation*}
\begin{split}
    0\leq & \displaystyle\int_\Omega\left(|\derivada u^*|^{p-2}\derivada u^*+ \mu(x)|\derivada u^*|^{q-2}\derivada u^*  \right)\derivada\Theta\,\mathrm{d}x\\
    & - \displaystyle\int_\Omega \left(a(x) (u^*)^{-\gamma}+\lambda(u^*)^{r-1} \right)\Theta\,\mathrm{d}x\\
    =&\displaystyle\int_{\left\{u^*+\epsilon h\geq 0\right\}}\left(|\derivada u^*|^{p-2}\derivada u^*+ \mu(x)|\derivada u^*|^{q-2}\derivada u^*  \right)\derivada(u^*+\epsilon h)\,\mathrm{d}x\\
    &-\displaystyle\int_{\left\{u^*+\epsilon h\geq 0\right\}}\left(a(x) (u^*)^{-\gamma}+\lambda(u^*)^{r-1} \right)(u^*+\epsilon h)\,\mathrm{d}x\\
    =&\norma[\derivada u]{p}{L^{p}(\Omega)}+ \norma[\derivada u]{q}{L_{\mu}^{q}(\Omega)}-\displaystyle\int_{\Omega}[a(x) (u^*)^{-\gamma}+\lambda(u^*)^{r-1}]\,\mathrm{d}x-\lambda \norma[u]{r}{{L^{r}(\Omega)}}\\
    &+\epsilon \displaystyle\int_\Omega\left(|\derivada u^*|^{p-2}\derivada u^*+ \mu(x)|\derivada u^*|^{q-2}\derivada u^*  \right)\derivada h\,\mathrm{d}x\\
    &-\epsilon\displaystyle\int_{\Omega}\left(a(x) (u^*)^{-\gamma}+\lambda(u^*)^{r-1} \right) h\,\mathrm{d}x + \displaystyle\int_{\left\{u^*+\epsilon h<0\right\}}\left(a(x) (u^*)^{-\gamma}+\lambda(u^*)^{r-1} \right)(u^*+\epsilon h)\,\mathrm{d}x\\
    & - \displaystyle\int_{\{u^*+\epsilon h<0\} }\left(|\derivada u^*|^{p-2}\derivada u^*+ \mu(x)|\derivada u^*|^{q-2}\derivada u^*  \right)\derivada(u^*+\epsilon h)\,\mathrm{d}x\\
       \leq & \,\, \epsilon \displaystyle\int_\Omega\left(|\derivada u^*|^{p-2}\derivada u^*+ \mu(x)|\derivada u^*|^{q-2}\derivada u^*  \right)\derivada h\,\mathrm{d}x\\
    & - \epsilon\displaystyle\int_{\Omega}\left(a(x) (u^*)^{-\gamma}+\lambda(u^*)^{r-1} \right) h\,\mathrm{d}x\\
    & - \displaystyle\int_{\{u^*+\epsilon h<0\} }\left(|\derivada u^*|^{p-2}\derivada u^*+ \mu(x)|\derivada u^*|^{q-2}\derivada u^*  \right)\derivada  h\,\mathrm{d}x,
\end{split}
\end{equation*}}
since $u^*\in \mathbf{N}_\lambda^+$ and $u^*\geq 0$. Note that the measure of the domain $ \{u^*+\epsilon h<0\}$ tend to zero as $\epsilon \rightarrow 0$. Therefore,
\begin{equation*}
    \displaystyle\int_{\{u^*+\epsilon h<0\} }\left(|\derivada u^*|^{p-2}\derivada u^*+ \mu(x)|\derivada u^*|^{q-2}\derivada u^*  \right)\derivada h\,\mathrm{d}x\rightarrow 0  
\end{equation*}
as $\epsilon\rightarrow 0.$

Dividing by $\epsilon$ and letting $\epsilon\rightarrow 0$, we have
\begin{equation*}
\begin{split}
    0\leq &\displaystyle\int_\Omega\left(|\derivada u^*|^{p-2}\derivada u^*+ \mu(x)|\derivada u^*|^{q-2}\derivada u^*  \right)\derivada h\,\mathrm{d}x\\
     &-\displaystyle\int_{\Omega}\left(a(x) (u^*)^{-\gamma}+\lambda(u^*)^{r-1} \right) h\,\mathrm{d}x.
\end{split}
\end{equation*}

Since $h\in \hzero(\Omega)$ is arbitrary chosen, equality must hold. It follows that $u^*$ is weak solution of problem Eq.(\ref{eq1}) such that $\mathbf{E}_\lambda(u^*)<0$ (see Propositions \ref{p32} and \ref{p34}).
\end{proof}

\begin{proposition}\label{p38} Suppose that  {\rm(\bf H)} is satisfied. Then there exists $\hat{\lambda}^*_0\in (0,\hat{\lambda^*}]$ such that $\mathbf{E}_\lambda|_{\mathbf{N}_\lambda^-}\geq 0$ for all $\lambda \in (0,\hat{\lambda^*_0}]$.
\end{proposition}

\begin{proof} From Proposition \ref{p34}, we know that $\mathbf{N}_\lambda^-\neq \emptyset$. Let $u\in \mathbf{N}_\lambda^-$. By the definition of $\mathbf{N}_\lambda^-$ and the embedding $\hzero(\Omega)\rightarrow L^{r}(\Omega)$, yields
\begin{equation*}
  \begin{split}
        \lambda(r+\gamma-1)\norma[u]{r}{L^r(\Omega)} &> (p+\gamma-1)\norma[\derivada u]{p}{L^p(\Omega)} + (q+\gamma-1)\norma[\derivada u]{q}{L_{\mu}^q(\Omega)}\\
        & \geq (p+\gamma-1)\norma[\derivada u]{p}{L^p(\Omega)}\\
        & \geq (p+\gamma-1)c_9^p\norma[u]{p}{L^r(\Omega)},
  \end{split}
\end{equation*}
for some $c_9>0$. Therefore, we get
\begin{equation}\label{eq330}
    \norma[u]{r}{L^r(\Omega)}\geq \left(\frac{c_9^p(p+\gamma-1)}{\lambda(r+\gamma-1)}\right)^{\frac{1}{r-p}}.
\end{equation}

Arguing by contradiction, let us suppose that the proposition statement is not true. Then we can find $u\in \mathbf{N}_\lambda^-$ such that $\mathbf{E}_\lambda u \leq 0$, that is,
\begin{equation}\label{eq331}
\frac{1}{p}\norma[\derivada u]{p}{L^p} + \frac{1}{q}\norma[\derivada]{q}{L_{\mu}^q} - \frac{1}{1-\gamma}\displaystyle\int_{\Omega}a(x) |u^*|^{1-\gamma}\mathrm{d}x - \frac{\lambda}{r}\norma[u]{r}{L^r(\Omega)} \leq 0.
\end{equation}

Since $\mathbf{N}_\lambda^-\subseteq \mathbf{N}_\lambda$, one has
\begin{equation}\label{eq332}
\frac{1}{q}\norma[\derivada]{q}{L_{\mu}^q} = \frac{1}{q}\displaystyle\int_{\Omega}a(x) |u^*|^{1-\gamma}\mathrm{d}x + \frac{\lambda}{q}\norma[u]{r}{L^r(\Omega)}-\frac{1}{q}\norma[\derivada u]{p}{L^p}.
\end{equation}

Substituting Eq.(\ref{eq332}) in Eq.(\ref{eq331}), yields 
\begin{equation*}
\left(\frac{1}{p}-\frac{1}{q}\right)\norma[\derivada u]{p}{L^p} +  \left(\frac{1}{q}-\frac{1}{\gamma-1}\right)\displaystyle\int_{\Omega}a(x) |u^*|^{1-\gamma}\mathrm{d}x + \lambda  \left(\frac{1}{q}-\frac{1}{r}\right)\norma[u]{r}{L^r(\Omega)}\leq 0,
\end{equation*}
and gives that
\begin{equation*}
\lambda \frac{r-q}{qr}\norma[u]{r}{L^r(\Omega)}\leq \frac{q+\gamma-1}{q(1-\gamma)}\displaystyle\int_{\Omega}a(x) |u^*|^{1-\gamma}\mathrm{d}x\leq \frac{q+\gamma-1}{q(1-\gamma)}c_{10}\norma[ u]{1-\gamma}{L^r},
\end{equation*}
for some $c_{10}>0$. Thus
\begin{equation}\label{eq333}
\norma[u]{r}{L^r(\Omega)}\leq c_{11}\left(\frac{1}{\lambda}\right)^{\frac{1}{r+\gamma-1}},
\end{equation}
for some $c_{11}>0$. Now we use Eq.(\ref{eq333}) in Eq.(\ref{eq330}), in order to obtain
\begin{equation*}
    c_{12}\left(\frac{1}{\lambda}\right)^{\frac{1}{r-p}}\leq c_{11}\left(\frac{1}{\lambda}\right)^{\frac{1}{r+\gamma-1}}, \mbox{ with } c_{12}= \left(\frac{c_9^p(p+\gamma-1)}{r+\gamma-1}\right)^{\frac{1}{r-p}}>0.
\end{equation*}

Consequently,
\begin{equation*}
    \begin{split}
        0<\frac{c_{12}}{c_{11}}&\leq \lambda^{\frac{1}{r-p}-\frac{1}{r+\gamma-1}}\\
        & = \lambda^{\frac{p+\gamma-1}{(r-p)(r+\gamma-1)}} \rightarrow 0 \mbox{ as } \lambda\rightarrow0^+,
    \end{split}
\end{equation*}
since $1<p<r$ and $\gamma\in (0,1)$, a contradiction. Thus, we can find $\hat{\lambda}^*_0\in (0, \hat{\lambda}^*]$ such that $\mathbf{E}_\lambda|_{\mathbf{N}_\lambda^-}> 0$ for all $\lambda \in (0,\hat{\lambda^*_0}]$.
\end{proof}

\begin{proposition}\label{p39} Suppose that {\rm(\bf H)} is satisfied and  consider $\lambda \in (0,\hat{\lambda}_0^* ]$. Then, there exists $v^*\in \mathbf{N}_\lambda^-$, with $v^*\geq 0$, such that
\begin{equation*}
    m_\lambda^- = \inf_{\mathbf{N}_\lambda^-}\mathbf{E}_\lambda = \mathbf{E}_\lambda(v^*)>0.
\end{equation*}
\end{proposition}

\begin{proof} The proof is similar to the proof of Proposition \ref{p34}. If $\{v_n\}_{n\in\mathbb{N}}\subset\mathbf{N}_\lambda^-\subset \mathbf{N}_\lambda $ is a minimizing sequence, from Proposition \ref{p31}, we know that $\{u_n\}_{n\in \mathbb{N}}\subset \hzero(\Omega)$ is bounded. Hence, we may assume that
\begin{equation*}
    v_n\rightarrow v^* \mbox{ in } \hzero(\Omega) \,\mbox{ and }\, v_n\rightarrow v^* \mbox{ in } L^r(\Omega).
\end{equation*}

Note that $v^{*}\neq 0$ by Eq.(\ref{eq330}). Now we use the point $t_2>0$ (see Eq.(\ref{eq312})), its follows that
\begin{equation*}
\Phi_{v^*}(t_2) = \lambda \norma[v^*]{r}{L^r(\Omega)} \,\mbox{ and } \,\Phi_{v^*}'(t_2)<0.
\end{equation*}

As in the proof of Proposition \ref{p34}, by applying Proposition \ref{p38}, we conclude that 
\begin{equation*}
v^*\in \mathbf{N}_\lambda^-, \, v^*\geq 0 \, \mbox{ and } \, m_\lambda^- = \mathbf{E}_\lambda(v^*)>0.
\end{equation*}
\end{proof}

\begin{proposition}\label{p310} Suppose that  {\rm(\bf H)} is satisfied and let $h\in\hzero(\Omega)$ and  $\lambda\in (0,\hat{\lambda}^*]$. Then there exists $b>0$ such that $\mathbf{E}_\lambda(v^*)\leq \mathbf{E}_\lambda(v^*+th)$, for all $t\in[0,b]$.
\end{proposition}

\begin{proof} The proof follows analogously to the one given in the Proposition \ref{p36}, replacing $u^*$ by $v^*$ in the definition of $\eta_h$ and using Lemma \ref{l35}.
\end{proof}

Now we have a second weak solution of problem Eq.(\ref{eq1}).

\begin{proposition}\label{p311} Let hypotheses {\rm(\bf H)} be satisfied and let $\lambda\in (0,\hat{\lambda}_0^*]$. Then $v^*$ is a weak solution of problem {\rm(\ref{eq1})} such that $\mathbf{E}_\lambda(v^*)>0$.
\end{proposition}

\begin{proof}
The proof can be done as the proof of Proposition \ref{p37}, by applying Propositions \ref{p310} and \ref{p39}.
\end{proof}

\begin{proof}{\bf (Theorem \ref{prin})}.
The proof follows of Propositions \ref{p310} and \ref{p311}.
\end{proof}
\section{Conclusion and future work}

\textcolor{blue}{The results obtained in this present paper certainly contribute directly to the theory of fractional operators, in particular, to fractional partial differential equations. However, there are still numerous issues that need to be addressed, and challenges that surround fractional operators involving functions with respect to $\psi(\cdot)$. Although the results here were successfully obtained, during the process, one of the challenges is to obtain estimates for the $\psi$-Hilfer fractional operator, which was not a simple task. However, as we have already discussed in other works with problems of fractional differential equations of the $p$-Laplacian type, some paths are already known. On the other hand, as highlighted in the introduction, here we work in the weightless $(\psi,\mathcal{H})$-fractional space, which does not allow us to discuss the particular case for $\psi(x)=x$. However, this is possible just by taking the space $\mathbb{H}_{\mathcal{H},0}^{\alpha,\beta;\psi}(\Omega)$ with weight. As a natural continuation of this work and with the aim of trying to expand and contribute to the area, there are some questions that motivate us to continue in this research project, as highlighted below, that is:}
\begin{itemize}
    \item 
    \textcolor{blue}{ Taking the problem Eq.(\ref{eq1}), we can work with the $p(x)$-Laplacian has a more complex nonlinearity that raises some of the essential difficulties, for example, it is inhomogeneous.}

    \item
    \textcolor{blue}{ Motivated by the problem Eq.(\ref{eq1}), we can infer conditions and make a connection with a Kirchhoff-type problem.}

\end{itemize}

\textcolor{blue}{We hope in the near future to start research addressing these issues. Little by little, we believe that this theory of fractional operators in the field of partial differential equations involving $p$-Laplacian is growing and being consolidated.}

\section{Declarations}

{\bf Conflict of interest} The authors declare that they have no conflict of interest.

\section*{Acknowledgements}
\textcolor{blue}{ All authors’ contributions to this manuscript are the same. All authors read and approved the final manuscript. We are very grateful to the anonymous reviewers for their useful comments that led to improvement of the manuscript.}



\begin{thebibliography}{}
	
	
	
\bibitem{Baroni} Baroni, P., M. Colombo, G. Mingione. Harnack inequalities for double phase functionals. Nonlinear Anal. 121 (2015), 206--222.


\bibitem{1Baroni} Baroni, P.,  Colombo, M.,   Mingione, G.  Non-autonomous functionals, borderline cases and related function classes, St. Petersburg Math. J. 27 (2016) 347–379.


\bibitem{Bahrouni} Bahrouni, A., V. D. Radulescu, P. Winkert. Double phase problems with variable growth and convection for the Baouendi-Grushin operator. Z. Angew. Math. Phys. 71(6) (2020), 183, 14 pp.



\bibitem{Colasuonno} Colasuonno, F., and M. Squassina. Eigenvalues for double phase variational integrals. Ann. Mat. Pura Appl. 4(6) 195 (2016), 1917--1959. 


\bibitem{Colombo} Colombo, M., G. Mingione. Bounded minimisers of double phase variational integrals. Arch.
	Ration. Mech. Anal. 218(1) (2015), 219--273.

\bibitem{Colombo1}  Colombo, M.,    Mingione, G.  Regularity for double phase variational problems, Arch. Ration. Mech. Anal. 215 (2015) 443–496.

\bibitem{Suwan} Suwan, I., M. Abdo, T. Abdeljawad, M. Mater, A. Boutiara, and M. Almalahi. Existence theorems for Psi-fractional hybrid systems with periodic boundary conditions AIMS. AIMS Math.; 7(1): 171- 186


\bibitem{Cui} Cui, N., and Hong-Rui Sun. Existence and multiplicity results for double phase problem with nonlinear boundary condition. Nonlinear Analysis: Real World Applications 60 (2021), 103307.


\bibitem{Gasinski} Gasínski, L., and N. S. Papageorgiou. Constant sign and nodal solutions for superlinear double phase problems. Adv. Calc. Var. 14(4) (2021), 613--626.

\bibitem{Ghanmi} Ghanmi, A., and Z. Zhang. Nehari manifold and multiplicity results for a class of fractional boundary value problems with $p$-Laplacian. Bull. Korean Math. Soc. 56(5) (2019), 1297--1314.

\bibitem{Hewitt} Hewitt, E. and K. Stromberg. Real and Abstract Analysis, Springer-Verlag, New York, 1965.

\bibitem{Jiao} Jiao, F., and Y. Zhou. Existence of solutions for a class of fractional boundary value problems via critical point theory. Comput. Math. Appl. 62(3) (2011), 1181--1199.


\bibitem{Srivastava} Srivastava, Hari M., and J. Vanterler da C. Sousa. Multiplicity of Solutions for Fractional-Order Differential Equations via the $\kappa(x)$-Laplacian Operator and the Genus Theory. Fractal and Fractional 6.9 (2022): 481.


\bibitem{Kilbas} Kilbas, A. A., Hari M. Srivastava, and Juan J. Trujillo. Theory and applications of fractional differential equations. Vol. 204. elsevier, 2006.

\bibitem{Lei} Lei, C.-Y. Existence and multiplicity of positive solutions for Neumann problems involving singularity and critical growth. J. Math. Anal. Appl. 459(2) (2018), 959--979.

\bibitem{Liu} Liu, W., and G. Dai. Existence and multiplicity results for double phase problem. J. Diff. Equ. 265.9 (2018), 4311--4334.

\bibitem{Machado} Machado, J. A. Tenreiro. The bouncing ball and the Grünwald-Letnikov definition of fractional operator. Frac. Calc. Appl. Anal. 24(4) (2021), 1003--1014.


\bibitem{Musi} Musielak, J. Orlicz Spaces and Modular Spaces. Springer-Verlag, Berlin, 1983.

\bibitem{Nemati} Nemati, S., Pedro M. Lima, and Delfim F. M. Torres. A numerical approach for solving fractional optimal control problems using modified hat functions. Commun. Nonlinear Sci. Numer. Simul. 78 (2019), 104849.


\bibitem{Norouzi} Norouzi, F., and Gaston M. N'Guérékata. A study of $\psi$-Hilfer fractional differential system with application in financial crisis. Chaos, Solitons \& Fractals: 6 (2021), 100056.

\bibitem{Norouzi23} Norouzi, F., and Gaston M. N'guérékata. Existence results to a $\psi$-Hilfer neutral fractional evolution equation with infinite delay. Nonautonomous Dyn. Sys. 8.1 (2021), 101--124.

\bibitem{Nyamoradi} Nyamoradi, N., and E. Tayyebi. Existence of solutions for a class of fractional boundary value equations with impulsive effects via critical point theory. Mediterr. J. Math. 15(3) (2018), 1--25.

\bibitem{Odibat} Odibat, Z., V. S. Erturk, P. Kumar, and V. Govindaraj. Dynamics of generalized Caputo type delay fractional differential equations using a modified Predictor-Corrector scheme. Phys. Scripta 96(12) (2021), 125213.

\bibitem{Ok} Ok, J. Partial regularity for general systems of double phase type with continuous coefficients.
Nonlinear Anal. 177 (2018), 673--698.

\bibitem{Papageorgiou} Papageorgiou, N. S., D. D. Repovs, and C. Vetro. Positive solutions for singular double phase problems. J. Math. Anal. Appl. 501.1 (2021), 123896.


\bibitem{Ragusa} Ragusa, M. A., A. Tachikawa. Regularity for minimizers for functionals of double phase with
variable exponents. Adv. Nonlinear Anal. 9(1) (2020), 710--728.


\bibitem{Silva1} Silva, C. J., and Delfim F. M. Torres. Stability of a fractional HIV/AIDS model. Math. Comput. Simul. 164 (2019), 180--190.


\bibitem{Sousa22} Sousa, J. Vanterler da C.,  Magnun N. N. dos Santos, E. da Costa, L. A. Magna and E. Capelas de Oliveira. A new approach to the validation of an ESR fractional model. Comput. Appl. Math. 40(3) (2021), 1--20.


\bibitem{Sousa8} Sousa, J. Vanterler da C., Leandro S. Tavares, and César E. Torres Ledesma. A variational approach for a problem involving a $\psi$-Hilfer fractional operator. J. Appl. Anal. Comput. 11.3 (2021), 1610--1630.



\bibitem{Sousa34} Sousa, J. Vanterler da C., Mouffak Benchohra, and Gaston M. N'Guérékata. Attractivity for differential equations of fractional order and $\psi$-Hilfer type. Frac. Cal. Appl. Anal. 23(4) (2020), 1188--1207.


\bibitem{Sousa5} Sousa, J. Vanterler da C., M. Aurora P. Pulido, and E. Capelas de Oliveira. Existence and Regularity of Weak Solutions for $\psi$-Hilfer Fractional Boundary Value Problem. Mediterr. J. Math. 18.4 (2021), 1--15. 

\bibitem{Sousa45} Sousa, J. Vanterler da C., L. S. Tavares, and E. Capelas de Oliveira. Existence and Uniqueness of mild and strong for fractional evolution equation. Palestine J. Math. 10.2 (2021).


\bibitem{Sousa1} Sousa, J. Vanterler da C., and E. Capelas De Oliveira. Leibniz type rule: $\psi$-Hilfer fractional operator. Commun. Nonlinear Sci. Numer. Simul. 77 (2019), 305--311.


\bibitem{Sousa6} Sousa, J. Vanterler da C. Nehari manifold and bifurcation for a $\psi$-Hilfer fractional $p$‐Laplacian. Math. Meth. Appl. Sci. (2021).


\bibitem{Sousa4} Sousa, J. Vanterler da C. Sousa, C. T. Ledesma, Mariane Pigossi and Jiabin Zuo. Nehari manifold for weighted singular fractional $\psi$-Laplace equations. (2021) prepint.

\bibitem{Sousa} Sousa, J. Vanterler da C., and E. Capelas De Oliveira. On the $\psi$-Hilfer fractional operator. Commun. Nonlinear Sci. Numer. Simul. 60 (2018), 72--91.




\bibitem{Sousa7} Sousa, J. Vanterler da C., Jiabin Zuo, and Donal O'Regan. The Nehari manifold for a $\psi$-Hilfer fractional $p$-Laplacian. Applicable Anal. (2021), 1--31.


\bibitem{Vangipuram} Vangipuram, L., and A. S. Vatsala. Basic theory of fractional differential equations. Nonlinear Analysis: Theory, Methods \& Applications 69(8) (2008), 2677--2682.

\bibitem{Wulong} Wulong, L., G. Dai, N. S. Papageorgiou, and P. winkert. Existence of solutions for singular double phase problems via the Nehari manifold method. arXiv:2101.00593 (2021).




\bibitem{You} You, Z., M. Fečkan, and JinRong Wang. Relative controllability of fractional delay differential equations via delayed perturbation of Mittag-Leffler functions. J. Comput. Appl. Math. 378 (2020), 112939.



\bibitem{Zhang} Zhang, Z., and J. Li. Variational approach to solutions for a class of fractional boundary value problem. Electronic J. Quali. Theory Diff. Equ. 2015(11) (2015), 1--10.

\bibitem{Zhikov} Zhikov, V. V. Averaging of functionals of the calculus of variations and elasticity theory. Izv. Akad. Nauk SSSR Ser. Mat. 50(4) (1986), 675--710.

\bibitem{Zhao} Zhao, Y., and L. Tang. Multiplicity results for impulsive fractional differential equations with $p$-Laplacian via variational methods. Boundary Value Probl. 2017(1) (2017), 1--15.

\bibitem{3Zhikov} Zhikov, V. V.  S.M. Kozlov, O.A. Oleinik, Homogenization of Differential Operators and Integral Functionals, Springer-Verlag, Berlin, 1994.

\bibitem{1Zhikov} Zhikov, V. V.  On Lavrentiev’s phenomenon, Russ. J. Math. Phys. 3 (1995) 249–269.

\bibitem{2Zhikov} Zhikov, V. V.  On some variational problems, Russ. J. Math. Phys. 5 (1997) 105–116.

\bibitem{Zhikov1} Zhikov, V. V. On variational problems and nonlinear elliptic equations with nonstandard growth conditions. J. Math. Sci. 173(5) (2011), 463--570.


\end{thebibliography}
\end{document}